\documentclass[a4paper,12pt]{amsart}
\usepackage{amsmath}
\usepackage{amssymb}
\usepackage{amscd}
\usepackage{url}
\usepackage{mathrsfs}
\usepackage[all]{xy}
\usepackage[dvipdfmx]{graphicx}
\usepackage{color}

\everymath{\displaystyle}
\theoremstyle{plain} 
\newtheorem{theorem}{\indent\sc Theorem}[subsection]
\newtheorem{lemma}[theorem]{\indent\sc Lemma}
\newtheorem{corollary}[theorem]{\indent\sc Corollary}
\newtheorem{proposition}[theorem]{\indent\sc Proposition}
\newtheorem{claim}[theorem]{\indent\sc Claim}

\theoremstyle{definition} 
\newtheorem{definition}[theorem]{\indent\sc Definition}
\newtheorem{condition}[theorem]{\indent\sc Condition}

\newtheorem{example}[theorem]{\indent\sc Example}

\begin{document}
\pagestyle{plain}
\thispagestyle{plain}

\title[On metrics with minimal singularities of line bundles]
{On metrics with minimal singularities of line bundles whose stable base loci admit holomorphic tubular neighborhoods}

\author[Genki HOSONO and Takayuki KOIKE]{Genki HOSONO$^{1}$ and Takayuki KOIKE$^{2}$}
\address{ 
$^{1}$ Graduate School of Mathematical Sciences\\
The University of Tokyo\\
3-8-1 Komaba, Meguro-ku\\
Tokyo, 153-8914\\
Japan
}
\email{genkih@ms.u-tokyo.ac.jp}
\address{
$^{2}$ Graduate School of Science, Osaka City University\\
3-3-138, Sugimoto, Sumiyoshi-ku Osaka, 558-8585\\
Japan
}
\email{tkoike@sci.osaka-cu.ac.jp}

\begin{abstract}
We investigate the minimal singularities of metrics on a big line bundle $L$ over a projective manifold when the stable base locus $Y$ of $L$ is a submanifold of codimension $r\geq 1$. 
Under some assumptions on the normal bundle and a neighborhood of $Y$, we give a explicit description of the minimal singularity of  metrics on $L$.
We apply this result to study a higher (co-)dimensional analogue of Zariski's example, in which the line bundle $L$ is not semi-ample, however it is nef and big. 
\end{abstract}


\maketitle


\section{Introduction}

The purpose of this paper is to investigate {\it metrics with minimal singularities} on a big line bundle $L$ on a projective manifold $X$. 
Metrics with minimal singularities have been introduced in \cite[Definition 1.4]{DPS01} as a weak analytic analogue of the so-called Zariski decomposition. 
There exists a metric with minimal singularities  uniquely up to certain equivalence of singularities when $L$ is pseudo-effective \cite[Theorem 1.5]{DPS01}. 
Indeed, the {\it equilibrium metric} $h_e$ of any $C^\infty$ Hermitian metric $h$ on $L$ has minimal singularities (see Example \ref{ex:equilibrium}).

On a higher-dimensional variety, a line bundle does not necessarily admit the Zariski decomposition. 
Nakayama constructed an example of a line bundle which admits no Zariski decomposition even after any modification {\cite[IV, \S 2.6]{N}}.
Nakayama's example is constructed as the relative tautological bundle on certain projective {space} bundle over an abelian variety.
Boucksom \cite{B} posed a decomposition called {\it divisorial Zariski decomposition}, in which the negative part of a big line bundle $L$ is identified with the divisorial (i.e.\ one-codimensional) part of the singularities of a metric with minimal singularities on $L$. 
From this point of view, it is important for a study of the Zariski decomposition to investigate the higher-codimensional part of the singularities of metrics with minimal singularities in detail.
In \cite{K1}, the second author explicitly described the metrics with minimal singularities for Nakayama's example we mentioned above. 

We also investigate the case where the line bundle $L$ is nef (and big) and thus $L$ has no negative part in the sense of Zariski decompositions.
In this case, our main interest is in the semi-positivity of the line bundle, i.e.\ whether $L$ admits a $C^\infty$ metric with semi-positive curvature or not. 
In \cite{K2}, the second author studied the metrics with minimal singularities on a line bundle called Zariski's example, which is known to be nef and big, but not semi-ample.
As a result, it was shown that Zariski's example admits a $C^\infty$ Hermitian metric with semi-positive curvature.

In this paper, we investigate the metrics of $L$ with minimal singularities for more general cases than both \cite{K1} and \cite{K2}.
Our main result has the following application:

\begin{theorem}\label{cor:higher_Zariski}
Take two general quadric surfaces $Q_1$ and $Q_2$ in $\mathbb{P}^3$ and fix general $N$ points $p_1, p_2, \dots, p_N$ in $Q_1\cap Q_2$ ($N\geq 12$). 
Denote by $\pi\colon X:={\rm Bl}_{\{p_1, p_2, \dots, p_N\}}\mathbb{P}^3\to \mathbb{P}^3$ the blow-up of 
$\mathbb{P}^3$ at these $N$ points, 
and by $D_1$ and $D_2$ the the strict transform of $Q_1$ and $Q_2$, respectively. 
Define a line bundle $L$ by $L:=\pi^*\mathcal{O}_{\mathbb{P}^3}(1)\otimes \mathcal{O}_X(D_1)$.
Then, the local weight function $\varphi_{{\rm min}, L}$ of { a metric with minimal singularities $h_{{\rm min}, L}$} of $L$ { (i.e. $\varphi_{{\rm min}, L}$ is a locally defined function such that $h_{{\rm min}, L}=e^{-\varphi_{{\rm min}, L}}$)} is written as 
\[
\varphi_{{\rm min}, L}(z, y)=\frac{N-12}{N-8}\cdot \log (|z_1|^2+|z_2|^2) +O(1) 
\]
on a neighborhood of { every}  point of $Y:=D_1\cap D_2$, where $y$ is a coordinate of $Y$ and $z=(z_1, z_2)$ is a system of local defining functions of $Y$.
We have that $\varphi_{{\rm min}, L}$ is locally bounded on $X \setminus Y$.
\end{theorem}

When $N=12$, the line bundle $L$ in this theorem is nef and big, but not semi-ample. Hence it can be regarded as a higher-dimensional analogue of Zariski's example. In this case, we will show that $L$ admits a $C^\infty$ Hermitian metric with semi-positive curvature (see \S \ref{section:zariski} for detail), which can be regarded as a two-codimensional analogue of \cite[Theorem 1.1]{K2}.

In what follows, $L$ denotes a big line bundle on a projective manifold $X$.
We study the metric with minimal singularities when $(X,L)$ satisfies the following condition.

\begin{condition}\label{cond:main}
(i) The stable base locus $Y={\bf B}(L)$ of $L$ is a smooth (i.e. non-singular) compact subvariety of codimension $r\geq 1$, \\
(ii) the normal bundle $N_{Y/X}$ of $Y$ admits a direct sum decomposition $N_{Y/X}=N_1\oplus N_2\oplus\cdots\oplus N_r$ into $r$ negative line bundles. 
\end{condition}
Nakayama's example satisfies Condition \ref{cond:main}.
For the pair $(X, L)$ with Condition \ref{cond:main}, we define the following convex set:
\[
\Box_L=\left\{\alpha=(\alpha_1, \alpha_2, \dots, \alpha_r)\in \mathbb{R}^r_{\geq 0} \left|
\begin{matrix}
 |\alpha| \leq 1,\text{ and} \\
 c_1(L|_Y)+\textstyle\sum_{\lambda=1}^r\alpha_\lambda c_1(N_\lambda^{-1})
 \text{ { is}  pseudo-effective}
\end{matrix}
\right.\right\}, 
\]
where $|\alpha|:=\textstyle\sum_{\lambda=1}^r\alpha_\lambda$. 
In \cite{K1}, a metric with minimal singularities on Nakayama's example was described explicitly by using $\Box_L$ on a neighborhood of $Y$.
It is easily observed that a metric with minimal singularities is locally bounded on the complement $X\setminus Y$ of $Y$ (see Example \ref{ex:bergman_type_sing_metric}).
Hence our interest is in the behavior of metrics with minimal singularities near $Y$. 
We always take a system of local defining functions $z=(z_1, z_2, \dots, z_r)$ of $Y$ so that, for each $\lambda$, the subbundle of $N_{Y/X}^{-1}$ generated by $dz_\lambda$ corresponds to the direct component $N_\lambda^{-1}$.

Our main result is stated as follows: 

\begin{theorem}\label{thm:main}
Let $X$ be a projective manifold, $L$ be a big line bundle on $X$, and $Y={\bf B}(L)$ be the stable base locus of $L$. 
Assume that $Y$ admits a holomorphic tubular neighborhood (see below) and $(X,L,Y)$ satisfies Condition \ref{cond:main}.
Assume also that $L|_Y\otimes N_\lambda^{-1}$ and $K_Y^{-1}\otimes N_\lambda^{-1}$ are positive for every $\lambda=1, 2, \dots, r$. 
Take $C^\infty$ Hermitian metrics $h_{L|_Y}$ on $L|_Y$ and $h_{N_\lambda}$ on $N_\lambda$ satisfying $\Theta_{h_{L|_Y}\otimes h_{N_\lambda}^{-1}}>0$ for every $\lambda$. 
Then the local weight function $\varphi_{{\rm min}, L}$ of { a metric with minimal singularities} $h_{{\rm min}, L}$  is written as 
\[
\varphi_{{\rm min}, L}(z, y)=\log \max_{\alpha \in \Box_L} \left(\prod_{\lambda=1}^r |z_\lambda|^{2\alpha_\lambda}\right)\cdot e^{(\varphi_\alpha)_e (y)}+ O(1)
\]
on a neighborhood of any given point of $Y$, where 
{we are formally regarding $0^0$ as $1$, }
$y$ is a coordinate of $Y$, $z=(z_1, z_2, \dots, z_r)$ is a system of local defining functions of $Y$ as  above, and $(\varphi_\alpha)_e$ is the local weight function of the equilibrium metric of $h_{L|_Y}\otimes h_{N_1}^{-\alpha_1}\otimes h_{N_2}^{-\alpha_2}\cdots\otimes h_{N_r}^{-\alpha_r}$ (see \S 2 for the notion of the ``metric'' $h_{L|_Y}\otimes h_{N_1}^{-\alpha_1}\otimes h_{N_2}^{-\alpha_2}\cdots\otimes h_{N_r}^{-\alpha_r}$ for real $\alpha_\lambda$'s).
\end{theorem}

A complex submanifold $Y \subset X$ is said to have a {\it holomorphic tubular neighborhood} if there exist a neighborhood $V$ of $Y$ in $X$, a neighborhood $\widetilde{V}$ of the zero section in $N_{Y/X}$ and a biholomorphism $i\colon V\to \widetilde{V}$ such that $i|_Y$ coincides with the natural isomorphism. 
Note that the description of the singularity of $\varphi_{{\rm min}, L}$ as in Theorem \ref{thm:main} does not depend on the choice of the coordinates (up to $O(1)$).
This theorem is a generalization of the main result of \cite{K1}. Moreover it is also a generalization of \cite{K2} in higher codimensional cases. 
In this theorem, Condition \ref{cond:main} and the condition that $Y$ admits a holomorphic tubular neighborhood are essential and can not be dropped (see \S \ref{section:BEGZ}). 

When $Y$ is an abelian variety (as in Nakayama's example), we have a sufficient condition for the existence of a holomorphic tubular neighborhood of $Y$ by Grauert's theory on a neighborhood of an exceptional subvariety \cite{G} (see \S \ref{section:tubular}).
As a result, we have the following theorem:

\begin{theorem}\label{cor:ab}
Let $X$ be a projective manifold, let $L$ be a big line bundle on $X$, and let $Y={\bf B}(L)$ be the stable base locus of $L$. 
Assume Condition \ref{cond:main}. 
Assume also that $Y$ is an abelian variery, $L|_Y\otimes N_\lambda^{-1}$ is positive for every $\lambda=1, 2, \dots, r$, and that $N_\lambda\cong N_\mu$ for every $\lambda$ and $\mu$. 
Then the local weight function $\varphi_{{\rm min}, L}$ of {a metric with minimal singularities $h_{{\rm min}, L}$} is written as 
\[
\varphi_{{\rm min}, L}(z, y)=\log \max_{\alpha \in \Box_L}\prod_{\lambda=1}^r |z_\lambda|^{2\alpha_\lambda} +O(1)
\]
on a neighborhood of any given point of $Y$, where $y$ is a coordinate of $Y$ and $z=(z_1, z_2, \dots, z_r)$ is a system of local defining functions of $Y$ as  above. 
\end{theorem}

When $L|_Y$ is pseudo-effective, it is natural to ask whether $(h_{{\rm min}, L})|_Y$ is a metric with minimal singularities on $L|_Y$. 
For $(X, L, Y)$ in Theorem \ref{thm:main}, it follows by definition that the convex set $\Box_L$ includes the origin $0$ if $L|_Y$ is pseudo-effective. 
In this case, it is directly deduced from Theorem \ref{thm:main} that $h_{{\rm min}, L}|_Y\leq (h_{L|_Y})_e \cdot e^{O(1)}$, which means that $h_{{\rm min}, L}|_Y$ has minimal singularities. 
Therefore we have the following: 

\begin{corollary}\label{cor:restr_min}
Let $X, L$, and $Y$ be those in Theorem \ref{thm:main}. 
Assume that $L|_Y$ is pseudo-effective. 
Then $h_{{\rm min}, L}|_Y$ is a singular Hermitian metric of $L|_Y$ with minimal singularities. 
\end{corollary}

The proof of Theorem \ref{thm:main} is based on the arguments in \cite{K2}. 
We first study a special case where $X$ is a projective space bundle over $Y$ and $L$ is the relative tautological bundle. 
After that, we apply the exact description of metrics with minimal singularities for this spacial case to the study of general $(X, L, Y)$ by using, what we call, the maximum construction technique (here we use the assumption of a holomorphic tubular neighborhood). 

The organization of the paper is as follows.
In \S 2, we introduce fundamental notation and recall some facts on projective {space} bundles and singular Hermitian metrics. 
In \S 3, we show the main result in the special case where $X$ is the total space of a projective {space} bundle. 
In \S 4, we prove Theorem \ref{thm:main} in general. 
In \S 5, we give a sufficient condition for the existence of a holomorphic tubular neighborhood by using Grauert's theory. Here we also show Theorem \ref{cor:ab}. 
In \S 6, we give several examples. 

\section{Preliminaries}

\subsection{Notations on projective {space} bundles}\label{subsection:notation_P_bdl}
Let $Y$ be a compact complex manifold. 
Let $M_1, M_2, \ldots, M_r$ and $M_{r+1}$ be holomorphic line bundles on $Y$.
Let $\{U_j\}_j$ be an open cover of $Y$.
Assume that every $U_j$ is sufficiently small so that $M_\lambda|_{U_j}$ is trivial for every $\lambda$ and $j$.
Then there exist local holomorphic trivializations given by sections $s_{j,\lambda} \in H^0(U_j,M_\lambda)$.
Denote by $E$ the vector bundle $M_1 \oplus M_2 \oplus \cdots \oplus M_{r+1}$.
Then $(s_{j,\lambda})_\lambda$ forms a holomorphic frame of $E$ on $U_j$.
Let $(\xi_{j,\lambda})_\lambda$ be the dual frame of $(s_{j,\lambda})_\lambda$.

We fix the notation on $\mathbb{P}(E)$ as follows.
Let us denote by $\mathbb{P}(E)$ the projective {space} bundle {\it of hyperplanes} of $E$ over $Y$, i.e.\ $\mathbb{P}(E):=\textstyle\bigcup_y (E^*_y \setminus 0)/\mathbb{C}^*$. We will denote the bundle {\it of lines} by ${\bf P}(E)$ in this paper.
Let $\pi$ denote the natural projection $\mathbb{P}(E) \to Y$.
We will use the notation $\mathbb{P}(E)|_{U_j}$ to denote $\pi^{-1} (U_j)$.
By using homogeneous coordinates,  $([x_{j,1}:x_{j,2}:\cdots:x_{j,r+1}],y)$ denotes the point $[x_{j,1}\xi_{j,1} + x_{j,2}\xi_{j,2}+ \cdots + x_{j,r+1}\xi_{j,r+1}] \in\mathbb{P}(E)_y$ on $\mathbb{P}(E)|_{U_j}$. Here $\mathbb{P}(E)_y$ denotes the fiber $\pi^{-1}(y)$.
Let $U_j^{(\lambda)}$ be an open set $\{([x_{j,1}\xi_{j,1} + x_{j,2}\xi_{j,2} + \cdots + x_{j,r+1}\xi_{j,r+1}], y)\mid y\in U_j,\, x_{j,\lambda} \neq 0 \}$ of $\mathbb{P}(E)$.
Note that $\{U_j^{(\lambda)}\}_{j,\lambda}$ forms an open cover of $\mathbb{P}(E)$.
The {\it tautological line bundle} $\mathcal{O}_{\mathbb{P}(E)}(1)$ on $\mathbb{P}(E)$  is defined by setting its fiber on $([\xi], y)$ as $E_y/{\rm Ker}\, \xi$, where $\xi$ denotes an element of $E^*_y \setminus 0$.
Let $\Gamma_\lambda$ be the divisor of $\mathbb{P}(E)$ defined as $\mathbb{P}(M_1 \oplus M_2 \oplus \cdots \widehat{M_\lambda} \cdots  \oplus M_{r+1})$. 
The following fact is obtained by a simple computation. 

\begin{lemma}\label{lem:Gamma}
$(i)$ 
$[\Gamma_\lambda] \otimes \pi^* M_\lambda = \mathcal{O}_{\mathbb{P}(E)}(1)$,
where $[\Gamma_\lambda]$ denotes the line bundle defined by the divisor $\Gamma_\lambda$.\\
$(ii)$ 
$N_{Y/X}= \textstyle\bigoplus_{\lambda=1}^r \mathcal{O}_{\mathbb{P}(E)}(1)|_Y \otimes M_\lambda^{-1}$.
\end{lemma}



\subsection{Singular Hermitian metrics}\label{subsection:shm}

In this subsection, we review some properties of singular Hermitian metrics on line bundles.

\begin{definition}
	Let $X$ be a (possibly non-compact) complex manifold and let $L$ be a line bundle on $X$.
	A {\it singular Hermitian metric} $h$ on $L$ is defined as a  metric of $L$ with the form $\|s\|^2_h = |s|^2 e^{-\phi}$ on $U$ for each trivialization $L|_U \cong U \times \mathbb{C}$, where $\phi \in L^1_{\rm loc}(U)$.
	In this situation, we will write as $h = e^{-\phi}$ and call $\phi$ as a {\it local weight}. Note that $\phi$ is a collection of a function defined on small open sets.
	The {\it curvature} of a singular Hermitian metric $h = e^{-\phi}$ is defined as a (1,1)-current $\Theta_h = \sqrt{-1}\partial \overline{\partial} \phi$.
\end{definition}

A singular Hermitian metric $h = e^{-\phi}$ is {\it semi-positively curved} (or $h$ admits {\it semi-positive curvature}) if its local weight $\phi$ is plurisubharmonic on the set where $\phi$ is defined. In this case, its curvature is non-negative as a (1,1)-current.

Let $h_1$ and $h_2$ be singular Hermitian metrics on $L$. We say that $h_1$ is {\it more singular} than $h_2$ when, for every relatively compact set $U$, there is a constant $C>0$ such that the inequality $h_1 \geq Ch_2$ holds on 
$U$. In this case we write $h_1 \gtrsim_{\rm sing} h_2$.
We say that $h_1$ and $h_2$ have {\it equivalent singularities} (written $h_1 \sim_{\rm sing} h_2$) when both $h_1 \lesssim_{\rm sing} h_2$ and $h_1 \gtrsim_{\rm sing} h_2$ hold.
A semi-positively curved singular Hermitian metric $h$ on $L$ has {\it minimal singularities} if $h \lesssim_{\rm si
ng} h'$ for any semi-positively curved singular Hermitian metric $h'$.
When $X$ is compact, $h \gtrsim_{\rm sing} h'$ holds if and only if there exists a constant $C$ such that $h_1 \geq C h_2$ on $X$.

To investigate singular Hermitian metrics, it will be convenient to consider globally defined functions corresponding to their local weights. 
For this reason, we introduce the notion of $\theta$-plurisubharmonic functions here.
Let $\theta$ be a smooth real (1,1)-form.
We say that a function $u \in L^1_{\rm loc} (X)$ is a {\it $\theta$-plurisubharmonic function} when the inequality $\theta + i\partial \overline{\partial} u \geq 0$ holds as currents. We denote the set of $\theta$-plurisubharmonic functions on $X$ by ${\rm PSH}(X,\theta)$.

Let $L$ be a holomorphic line bundle on $X$.
Fix a smooth Hermitian metric $h_0$ on $L$ with curvature $\theta$.
Then there is a one-to-one correspondence between $\theta$-plurisubharmonic functions $u$ and semi-positively curved singular Hermitian metrics $h_0 \cdot e^{-u}$ on $L$.
We define $\theta$-plurisubharmonic functions with minimal singularities similarly to the case of metrics. 
Namely, a $\theta$-plurisubharmonic function $u$ has {\it minimal singularities} (in ${\rm PSH}(X,\theta)$) if, for every $\theta$-plurisubharmonic function $u'$, there exists a (local) constant $C$ such that $u \geq u' + C$ on each compact set.
For an $\mathbb{R}$-line bundle $L$ (i.e.\ a formal ``line bundle'' corresponding to an $\mathbb{R}$-divisor), a notion of singular Hermitian metric on $L$ is well-defined formally in this sense.

\begin{example}\label{ex:equilibrium}
Assume $X$ is compact and $L$ is {\it pseudo-effective}, i.e.\ $L$ admits a semi-positively curved singular Hermitian metric.
Fix a smooth metric $h$ with curvature $\theta$. 
{ Then, the function defined by
\[V_\theta:=\sup\{v \in {\rm PSH}(X,\theta) \mid v \leq 0 \}\]
is $\theta$-plurisubharmonic.
}
It is easily observed that $V_\theta$ { has} minimal singularities { in $PSH(X,\theta)$}.
The corresponding singular Hermitian metric $h\cdot e^{-V_\theta}$ is denoted by $h_e$, which is called the {\it equilibrium metric}.
\end{example}

\begin{example}\label{ex:bergman_type_sing_metric}
	Fix a smooth Hermitian metric $h_0$ on $L$.
	Let $f_1, f_2, \ldots, f_N \in H^0(X,L)$ be global holomorphic sections of $L$.
	Then we define a singular Hermitian metric $h$ by the formula
	\[\|f\|^2_h := \frac{\|f\|^2_{h_0}}{\sum_{j=1}^{N} \|f_j\|^2_{h_0}}.\]
	In this manner, we obtain a semi-positively curved singular Hermitian metric $h$ which is smooth on the Zariski open set $\textstyle\bigcup_{j=1}^N\{f_j \neq 0 \}$.
	If $L$ is big, { there exist a} finite number of sections $f_1, \ldots, f_N \in H^0(X,L^m)$ for sufficiently large $m$ such that $\{f_1 = \cdots = f_N = 0\} = {\bf B}(L)$ (\cite[2.1.21]{Laz}). Here we denote the tensor product $L^{\otimes m}$ by $L^m$.
	Then, we have a singular Hermitian metric on $L^m$ which is smooth on $X \setminus {\bf B}(L)$. By taking $m$-th root, we can define a singular Hermitian metric on $L$ (we call it a {\it Bergman-type metric} on $L$ obtained by $f_1,\ldots,f_N$).
\end{example}

\begin{example}\label{ex:Bergman-minimal}
Let $X$ be a compact complex manifold and { let} $L$ be a line bundle.
Fix a smooth volume form $dV$ on $X$ and a smooth metric $h = e^{-\phi}$ on $L$.
Let $\theta$ be the curvature of $h$. We define a $\theta$-{plurisubharmonic} function $V_{\phi,B}$ by
\[V_{\phi, B} := V_{h,B} :=\sup \left\{ \frac{1}{m} \log|f|^2_{h^m} \middle| m \in \mathbb{Z}, f \in H^0(X,L^m),\int_{X}|f|^2_{h^m} dV \leq 1 \right\}.\]
The corresponding singular Hermitian metric on $L$ and its local weight are denoted by $h_B = e^{-\phi_B}$.
By Proposition \ref{prop:e-vs-b} { below}, $h_B$ { has minimal singularities} when $L$ is big.

We use this construction when $L$ is a $\mathbb{Q}$-line bundle with a smooth metric $h = e^{-\phi}$, that is, for some integer $m > 0$, $L^m$ is an ordinary line bundle and $h^m = e^{-m\phi}$ is a smooth metric on $L^m$.
In this case, we take the smallest integer $m>0$ such that $L^m$ is a $\mathbb{Z}$-line bundle { and} define $\phi_B$ by $(1/m)(m\phi)_B$.
\end{example}

To compare the metrics $h_e$ and $h_B$, we need the following proposition.

\begin{proposition}[{\cite[Lemma 2.10]{K2}}]\label{prop:e-vs-b}
	Let $X$ be a projective manifold and { let} $L$ be a big line bundle.
	Let $h = e^{-\phi}$ be a smooth Hermitian metric on $L$.
	Fix a smooth volume form $dV$ on $X$.
	Then, there is a constant $C$ such that the inequality
	\[V_{\phi,B} - C \leq V_\theta \leq V_{\phi,B} \]
	holds.
\end{proposition}

Before starting the proof, we shall explain how we use Proposition \ref{prop:e-vs-b} in \S 3. We shall apply it to a family of $\mathbb{Q}$-line bundles of the form
\[L^\alpha := L_1^{\alpha_1} \otimes L_2^{\alpha_2} \otimes \cdots \otimes L_{r+1}^{\alpha_{r+1}},\]
where $L_\lambda$ are $\mathbb{Z}$-line bundles, $\alpha_\lambda \geq 0$ and $\alpha_1 + \cdots + \alpha_{r+1} = 1$.
Let $e^{-\phi_\lambda}$ be a fixed smooth metric on $L_\lambda$ and { let} $m$ be the smallest positive integer such that $(L^\alpha)^m$ is a $\mathbb{Z}$-line bundle.
Then, the local weight $m\phi_\alpha := m\textstyle\sum_\lambda \alpha_\lambda \phi_\lambda$ defines a smooth metric on $(L^\alpha)^m$.
The constant $C$ in Proposition \ref{prop:e-vs-b} depends only on $C_1$, $C_2$ and $C(\phi)$, which will be defined in the proof below.
We note that { constants} $C_1$ and $C_2$ are independent of the choice of line bundles, and $C(\phi)$ only depends on the differences $(\textstyle\sup_{B_j''} - \textstyle\inf_{B_j''}) \phi$, where { $\{B_j''\}_j$ is a open cover of $X$ consisting of open balls}.
Thus there exists a constant $C_3$ depending on the metrics $e^{-\phi_\lambda}$ and independent of $\alpha$, such that
\[V_{m\phi^\alpha,B} - \log(C_1C_2) - mC_3 \leq V_{m\theta_\alpha} \leq V_{m\phi^\alpha,B}.\]
Dividing by $m$, we have that 
\[V_{\phi^\alpha,B} - \frac{1}{m}\log(C_1C_2) - C_3 \leq V_{\theta_\alpha} \leq V_{\phi^\alpha,B}.\]
In conclusion, there exists a constant $C$ such that we have 
\begin{equation}
V_{\phi_\alpha,B} - C \leq V_{\theta_\alpha} \leq V_{\phi_\alpha,B}\label{eqn:V_compare}
\end{equation}
for every $\alpha$ such that $\alpha_\lambda \geq 0$ and $\alpha_1 + \alpha_2 \cdots + \alpha_{r+1} = 1$.

\begin{proof}[\sc Proof of Proposition \ref{prop:e-vs-b}]
First we prove the inequality $V_\theta \leq V_{\phi,B}$.
Since $L$ is big, there exits a singular Hermitian metric $\psi_+$ on $L$ such that its curvature is a K\"{a}hler current, i.e.\ $\Theta_{\psi_+} \geq \epsilon \omega$ for some $\epsilon>0$ and some K\"{a}hler form $\omega$.
Define a $\theta$-{ plurisubharmonic} function $V_+$ by $\psi_+ = \phi + V_+$.
We { may} assume that $V_+ \leq 0$.
Let
$V_\ell := \left(1-{1}/{\ell}\right)V_\theta + ({1}/{\ell})V_+$
and $\phi_\ell := \phi + V_\ell$.
Then the curvature of the metric $e^{-\phi_\ell}$ is a K\"{a}hler current.
Now consider the following approximations:
\[V_{\phi, B,m} := \sup{}^* \left\{\frac{1}{m} \log |f|^2_{m\phi}\middle| f \in H^0(X, L^m), \int_X |f|^2 e^{-m\phi} dV \leq 1 \right\}, \text{ and}\]
\[V_{\phi_\ell, B,m} := \sup{}^* \left\{\frac{1}{m} \log |f|^2_{m\phi_\ell}\middle| f \in H^0(X, L^m), \int_X |f|^2 e^{-m\phi_\ell} dV \leq 1 \right\}. \]
Then we have that
$V_\ell + V_{\phi_{\ell},B,m}\leq V_{\phi,B,m}$.
Applying Demailly's approximation theorem (\cite[Theorem 13.21]{D2}) to $\phi_\ell$, we have that
$V_{\phi_\ell, B, m}\geq V_{\ell} - C/m$,
where $C$ is independent of $\ell$ and $m$.
Hence we have
$V_\ell - \textstyle\frac{C}{m} \leq V_{\phi,B,m} \leq V_{\phi,B}$.
By letting $m\to \infty$, we obtain $V_\ell \leq V_{\phi, B}$.
After that, letting $\ell \to \infty$, we have that $V_\theta \leq V_{\phi,B}$.

Next we prove the inequality $V_{\phi,B} - C \leq V_\theta$.
Fix a collection of open coordinate balls $B'_j \subset B''_j \subset B_j$ such that $\{B'_j\}_j$ is an open cover of $X$ and the radii of $B'_j$, $B''_j$ and $B_j$ are 1/2, 1 and 2 respectively.
Fix a local trivialization of $L$.
Take $f \in H^0(X, L^m)$ with $\int_X |f|^2 e^{-m\phi}dV \leq 1$.
Then, for every $p \in B'_j$, we have that
\begin{align*}
|f(p)|^2 \leq& \frac{1}{\pi^n (1/2)^{2n}/n! } \int_{|z-p| < 1/2} |f|^2 d\lambda \\
\leq & C_1 C_2 \cdot e^{m \sup_{B''_j} \phi} \int_{|z-p| < 1/2} |f|^2 e^{-m\phi} dV\\
\leq & C_1 C_2 \cdot e^{m \sup_{B''_j} \phi}.
\end{align*}
Here we write the constants as $C_1 := \textstyle\frac{1}{\pi^n (1/2)^{2n}/n! }$ and $C_2 := \textstyle\sup_{B''_j} d\lambda/dV$.
{ It follows} that $|f(p)|^2 e^{-m\phi(p)} \leq C_1C_2 \cdot \exp({m( \sup_{B''_j} \phi - \phi (p))}) \leq C_1 C_2 \cdot \exp({m(\sup_{B''_j} - \inf_{B''_j} )\phi})$.
{ Thus} we have that
$$\frac{1}{m}\log|f(p)|_{m\phi}^2 \leq (\log C_1C_2)/m + \left(\sup_{B''_j} - \inf_{B''_j}\right) \phi.$$
It follows that
$$V_{\phi,B,m}(p) \leq (\log C_1C_2)/m + \left(\sup_{B''_j} - \inf_{B''_j}\right) \phi.$$
The right-hand side is estimated by using the constants $C_1, C_2$ and a constant $C(\phi)$ depends only on $\phi$. Taking the supremum over $m$, we have that 
$V_{\phi,B}(p) \leq \log (C_1 C_2) + C(\phi)$. { We denote this constant by $C$.}
Considering all $B_j$, we have $V_{\phi,B} - C \leq V_\theta$ for some constant $C$.
\end{proof}



\section{Projective bundles}\label{section:P_bdl}

\subsection{Settings in the case of $\mathbb{P}^r$-bundle}
Let $Y$ be a projective manifold. 
Let $M_1,M_2$, $\cdots$, $M_r$ and $M_{r+1}$ be line bundles.
We assume that the first $r$ line bundles $M_1,\ldots, M_r$ are ample (we do not assume the ampleness of $M_{r+1}$).
Define a manifold $X$ by $X:=\mathbb{P}(M_1 \oplus M_2 \oplus \cdots \oplus M_{r+1})$ and a line bundle $L$ on $X$ by $L:=\mathcal{O}_{\mathbb{P}(M_1 \oplus M_2 \oplus \cdots \oplus M_{r+1})}(1)$.
Let us recall that $\mathbb{P}(E)$ denotes the projective {space} bundle of {\it hyperplanes} of $E$.
Let $\pi$ denote the natural projection $X \to Y$.
We regard $Y$ as a submanifold of $X$ via the inclusion { $\mathbb{P}(M_{r+1}) \subset \mathbb{P}(M_1 \oplus M_2 \oplus \cdots \oplus M_{r+1} )$} induced by the projection $M_1 \oplus M_2 \oplus \cdots \oplus M_{r+1} \to M_{r+1}$.
Let $h_\lambda$ ($1 \leq \lambda \leq r+1$) be smooth { Hermitian} metrics on $M_\lambda$ and $\theta_\lambda$ be the curvature forms of $h_\lambda$.
Here we assume that every $h_\lambda$ ($1 \leq \lambda \leq r$) has a positive curvature, i.e.\ { the curvature form} $\theta_\lambda$ is a positive (1,1)-form for every $\lambda =1,\ldots,r$.
Let us denote by $h_{L} = e^{-\varphi_{L}}$ the naturally induced metric on $L$ from $h_1, \ldots, h_{r+1}$ by considering the Euler sequence.
We denote by $\theta_L$ the curvature of $h_L$.
Let $\Box_L$ be a convex set defined in \S 1 as follows: 
\[
\Box_L=\left\{\alpha=(\alpha_1, \dots, \alpha_r)\in \mathbb{R}^r_{\geq 0} \middle|
\begin{matrix}
  |\alpha| \leq 1,\text{ and} \\
  c_1(L|_Y)+\textstyle\sum_{\lambda=1}^r\alpha_\lambda c_1(M_\lambda \otimes L|_Y^{-1}) \text{ is pseudo-effective}
\end{matrix}
\right\}, 
\]
where $|\alpha|$ denotes $\alpha_1 + \alpha_2 + \cdots + \alpha_r$.
Here we use the direct decomposition $N_{Y/X}= \textstyle\bigoplus_{\lambda=1}^r (L \otimes \pi^*M_\lambda^{-1})|_Y$ (see Lemma \ref{lem:Gamma}). 

{ The following theorem is the main result of this section.}

\begin{theorem}\label{thm:P_bdl_main}
	Let $Y$, $M_\lambda$, $X$, $L$ and $h_\lambda$ be as above.
	For an $r$-tuple of { non-negative} real numbers $\alpha = (\alpha_1,\ldots, \alpha_{r})$ { with $\alpha_1 + \alpha_2 + \cdots + \alpha_r \leq 1$} and a real number $\alpha_{r+1} := 1 - \alpha_1 {- \alpha_2} - \cdots - \alpha_{r}$, define a function $u_\alpha(x)$ on $X$ as follows: $$u_\alpha (x) := \alpha_1 \log|s_1|^2_{\widehat{h}_1} + \alpha_2\log|s_2|^2_{\widehat{h}_2}+ \cdots + \alpha_{r+1} \log |s_{r+1}|^2_{\widehat{h}_{r+1}} +\pi^* V_{\theta_\alpha}.$$
	Here, $s_\lambda$ denotes the canonical section of a divisor $\Gamma_\lambda$, 
	$\widehat{h}_\lambda$ denotes the metric on the line bundle $[\Gamma_\lambda]$ defined by $h_L/\pi^*h_\lambda$, and $\theta_\alpha$ denotes the (1,1)-form $\textstyle\sum_{\lambda=1}^{r+1}\alpha_\lambda\theta_\lambda$.
	Then,\\
	$(i)$ $u_\alpha$ is a $\theta_L$-plurisubharmonic function.\\
	$(ii)$ For every fixed $x \in X$, there exists the maximum value of the function $u_\alpha (x)$ of $\alpha$ on $\alpha \in \Box_L$.\\
	$(iii)$ Define a function 
	$\widehat{V} (x)$ on $X$ by $\widehat{V} (x):={\max_{\substack{(\alpha_1, \ldots, \alpha_r) \in \Box_L  }}} u_\alpha (x)$.
	Then $\widehat{V}$ is upper semi-continuous and $\theta_L$-plurisubharmonic. \\
	$(iv)$ $\widehat{V}$ is a { $\theta_L$-plurisubharmonic function with minimal singularitie}s.
\end{theorem}

\subsection{The relation between Theorem 3.1 and Theorem 1.2.}
Before proving Theorem \ref{thm:P_bdl_main}, we shall explain the relation between Theorem \ref{thm:P_bdl_main} and Theorem \ref{thm:main}. We assume that $X$, $Y$ and $L$ are { those} in Theorem \ref{thm:main}.
We construct a ``projective {space} bundle model'' $(\widetilde{X}, \widetilde{Y}, \widetilde{L})$, to which we apply Theorem \ref{thm:P_bdl_main}.
We define $\widetilde{X}$ by $\mathbb{P}(\mathbb{I}_Y \oplus N^*_{Y/X})$, $\widetilde{Y}$ by $\mathbb{P}(\mathbb{I}_Y)$, and $\widetilde{L}$ by $\mathcal{O}_{\mathbb{P}(\mathbb{I}_Y \oplus N^*_{Y/X})} (1)$.
In \S 4, we use a trick called { the} {\it maximum construction} to get a metric with minimal singularities on $L$ from that on $\widetilde{L}$.

To check that $(\widetilde{X}, \widetilde{Y}, \widetilde{L})$ satisfies the assumption of Theorem \ref{thm:P_bdl_main}, we have to choose appropriate line bundles $M_\lambda$ on $\widetilde{Y}$ as in the following lemma.

\begin{lemma}
Let $N_\lambda$ be line bundles as in Theorem \ref{thm:main}.
If one take 
$M_\lambda = L|_Y \otimes N_\lambda^{-1} \text{ for } \lambda = 1,\ldots,r$ 
and 
$M_{r+1} = L|_Y$, one have that
$\widetilde{X} \cong \mathbb{P}(M_1 \oplus \cdots \oplus M_{r+1})$
and 
$\widetilde{L} = \mathcal{O}_{\mathbb{P}(M_1\oplus M_2\oplus\cdots\oplus M_{r+1})}(1)$.
\end{lemma}

\begin{proof}
We have that
\begin{align*}
\mathbb{P}(M_1 \oplus \cdots \oplus M_{r+1}) &= \mathbb{P}(L|_Y \otimes (N_1^{-1} \oplus \cdots\oplus N_r^{-1} \oplus \mathbb{I}_Y))\\
&\cong\mathbb{P}(N_1^{-1} \oplus \cdots\oplus N_r^{-1} \oplus \mathbb{I}_Y)\\
&=\mathbb{P}(N_{Y/X}^* \oplus \mathbb{I}_Y) = \widetilde{X}.
\end{align*}
We also have that 
$\mathcal{O}_{\mathbb{P}(M_1\oplus M_2\oplus\cdots\oplus M_{r+1})}(1)
=\mathcal{O}_{\mathbb{P}(N_1^{-1}\oplus N_2^{-1}\oplus\cdots\oplus N_r^{-1}\oplus \mathbb{I}_Y)}(1)\otimes \pi^*M_{r+1} = \widetilde{L}$. 
\end{proof}

Note that, under this choice, we have 
$N_{Y/X}= \textstyle\bigoplus_{\lambda=1}^r (L \otimes \pi^*M_\lambda^{-1})|_Y$. 
To use the maximum construction argument in \S 4, we have to use the following lemma.

\begin{lemma}
In this situation, one have $SB(\widetilde{L}) \subset \widetilde{Y}$.
\end{lemma}
\begin{proof}
Recall that $s_\lambda$ is the canonical section of the line bundle $[\Gamma_\lambda]= \pi^*M_\lambda^{-1}\otimes\widetilde{L}$ (Lemma \ref{lem:Gamma}).
For every global section $f\in H^0(Y, \pi^*M_\lambda^m)$ ($m \geq 0$), we have that $s_\lambda^m\otimes \pi^*f \in H^0(\widetilde{X}, \widetilde{L}^m)$.
By the assumption of Theorem \ref{thm:main}, the line bundle $L|_Y \otimes N_\lambda^{-1}$ is ample for { every} $\lambda = 1,\ldots,r$.
Therefore, for sufficiently large $m$, { there exist} global sections of $\widetilde{L}$ whose common zero is $\Gamma_\lambda$.
Using this argument for { every} $\lambda$, we have that  ${\bf B}(\widetilde{L})\subset \widetilde{Y}$. 
\end{proof}

\subsection{Proof of Theorem \ref{thm:P_bdl_main}}
\subsubsection{The outline of the proof}\label{subsection: outline}
We obtain $(i)$ easily from the construction of $u_\alpha$.
We will prove $(ii)$ and $(iii)$ in \S \ref{subsection: usc}.
Now we explain the outline proof of $(iv)$.
Fix a K\"{a}hler form $\omega$ on $Y$.
Define functions $\widehat{V}^\mathbb{Q}$ and  $\widehat{V}^\mathbb{Q}_B$ on $X$ by
\[\widehat{V}^\mathbb{Q}:=\hskip-4mm{\sup_{\substack{\alpha = (\alpha_1, \ldots, \alpha_r)\\ \alpha \in \Box_L   \cap \mathbb{Q}^r} }}\hskip-6mm\,^* \hskip3mm u_\alpha\]
and
\[\widehat{V}_B^\mathbb{Q}:=\hskip-4mm{\sup_{\substack{\alpha = (\alpha_1, \ldots, \alpha_r)\\ \alpha \in \Box_L   \cap \mathbb{Q}^r} }}\hskip-6mm\,^* \hskip3mm \left[\sum_{\lambda=1}^{r+1} \alpha_\lambda\log|s_\lambda|^2_{\widehat{h}_\lambda} +\pi^* V_{h_1^{\alpha_1} h_2^{\alpha_2}\cdots h_{r+1}^{\alpha_{r+1}},B}\right],\]
where $V_{h_1^{\alpha_1} h_2^{\alpha_2}\cdots h_{r+1}^{\alpha_{r+1}},B}$ is a function on $Y$ defined as in Example \ref{ex:Bergman-minimal} with respect to the volume form $\omega^n$ on $Y$.
In \S 3.3.4, we prove that $V_{\theta_L,B} \leq \widehat{V}^{\mathbb{Q}}_B+C$ holds for some constant $C$, where $V_{\theta_L,B}$ is also defined as in Example \ref{ex:Bergman-minimal} (we specify the volume form on $X$ later in \S 3.3.3).
Then we have that, for some $C' \geq 0$,
\[V_{\theta_L,B} \leq \widehat{V}^\mathbb{Q}_B + C \leq \widehat{V}^\mathbb{Q} + C' \leq \widehat{V}+C'.\]
Here, the second inequality follows from the equation (\ref{eqn:V_compare}) before the proof of Proposition \ref{prop:e-vs-b}.

\subsubsection{Proof of Theorem \ref{thm:P_bdl_main} $(ii)$ and $(iii)$}\label{subsection: usc}
In this subsection, we will show the upper semicontinuity of $\widehat{V}$.
For simplicity of notation, we write $V_\alpha$ instead of $V_{\theta_\alpha}=\sup\{\psi\in PSH(Y, \theta_\alpha)\mid \psi\geq 0\}$.
We will show the following proposition.

\begin{proposition}\label{prop:usc}
	The function $F\colon\Box_L\times Y\to \mathbb{R}\cup\{\infty\}: F(\alpha, y):=V_\alpha(y)$ is upper semi-continuous.
\end{proposition}

From this proposition and compactness of $\Box_L$, a standard argument shows (ii) and (iii) of Theorem \ref{thm:P_bdl_main}. 
In this subsection, we write $\alpha\leq \beta$ when $\alpha_\lambda\leq \beta_\lambda$ for every $\lambda$.
To prove Proposition \ref{prop:usc}, we need the following lemma.

\begin{lemma}\label{lem:V_alpha_lim}
	Let $\alpha$ and $\beta$ be points in $\Box_L$.\\
	$(i)$ If $\alpha\leq \beta$, $\textstyle\frac{V_\alpha}{1-|\alpha|}\leq \textstyle\frac{V_\beta}{1-|\beta|}$. \\
	$(ii)$ $\textstyle\lim_{\beta\downarrow\alpha}\textstyle\frac{V_\beta}{1-|\beta|}=\textstyle\frac{V_\alpha}{1-|\alpha|}$, where $\textstyle\lim_{\beta\downarrow\alpha}$ means the limit as $\beta$ approaches to $\alpha$ under the condition $\alpha\leq\beta$. 
\end{lemma}

\begin{proof}
	$(i)$ Consider the local weights
	$\varphi_\alpha:=\textstyle\sum_{\lambda=1}^r\alpha_\lambda\varphi_\lambda + (1-|\alpha|)\varphi_{r+1}$.
	First, we use the equation
	\[
	\varphi_\beta+\frac{1-|\beta|}{1-|\alpha|}\cdot V_\alpha
	=\frac{1-|\beta|}{1-|\alpha|}\cdot(\varphi_\alpha+V_\alpha)+\sum_{\lambda=1}^r \left(\beta_\lambda-\frac{1-|\beta_\lambda|}{1-|\alpha_\lambda|}\cdot \alpha_\lambda\right)\cdot\varphi_\lambda.
	\]
	As the right-hand side is plurisubharmonic, we have that $\textstyle\frac{1-|\beta|}{1-|\alpha|}\cdot V_\alpha$ is $\theta_\beta$-plurisubharmonic.
	Since this function is non-positive, we obtain
	\[
	\frac{1-|\beta|}{1-|\alpha|}\cdot V_\alpha\leq V_\beta
	\]
	by the definition of $V_\beta$.
	
	$(ii)$ Take a sequence $\{\beta^{(\nu)}\}_{\nu=1}^\infty$, $\beta^{(\nu)} = (\beta^{(\nu)}_1, \beta^{(\nu)}_2, \dots, \beta^{(\nu)}_r)$, with $\beta^{(\nu)}_\lambda\downarrow\alpha_\lambda$ for every $\lambda=1, 2, \dots, r$.
	We shall prove
	$$
	\lim_{\nu\to \infty}\frac{V_{\beta^{(\nu)}}}{1-|\beta^{(\nu)}|}=\frac{V_\alpha}{1-|\alpha|}.
	$$
	By $(i)$, the inequality 
	$
	\lim_{\nu\to \infty}\textstyle\frac{V_{\beta^{(\nu)}}}{1-|\beta^{(\nu)}|}=\textstyle\frac{V_\alpha}{1-|\alpha|}
	$
	holds. Hence it is sufficient to prove the converse inequality.  
	
	Let us consider the local weight
	\[
	\frac{1}{1-|\beta^{(\nu)}|}\cdot (\varphi_{\beta^{(\nu)}}+V_{\beta^{(\nu)}})
	=\varphi_{r+1}+\sum_{\lambda=1}^r\frac{\beta^{(\nu)}_\lambda}{1-|\beta^{(\nu)}|}\cdot\varphi_\lambda+\frac{V_{\beta^{(\nu)}}}{1-|\beta^{(\nu)}|}.
	\]
	By the right-hand side, this weight is clearly decreasing in $\nu$. Moreover, by focusing on the left-hand side, we have that this weight is plurisubharmonic. 
	Therefore the limit
	$
	{\varphi_{\alpha}}/({1-|\alpha|})+
	\lim_{\nu\to \infty}{V_{\beta^{(\nu)}}}/({1-|\beta^{(\nu)}|})
	$
	is also plurisubharmonic. As the function $(1-|\alpha|)\cdot\lim_{\nu\to \infty}\textstyle\frac{V_{\beta^{(\nu)}}}{1-|\beta^{(\nu)}|}$ is non-positive and $\theta_\alpha$-plurisubharmonic, we have that
	$$
	(1-|\alpha|)\cdot \lim_{\nu\to \infty}\frac{V_{\beta^{(\nu)}}}{1-|\beta^{(\nu)}|}
	\leq V_\alpha.
	$$
\end{proof}

\begin{proof}[Proof of Proposition \ref{prop:usc}]
	Fix a point $(\alpha^0, y^0) \in \Box_L \times Y$.
	We shall prove the upper semicontinuity of $F$ at $(\alpha^0, y^0)$.
	 
	First, we treat the case where $|\alpha^0|=1$.	
	It is sufficient to prove
	\[
	\limsup_{\Box_L\times Y\ni(\alpha, y)\to (\alpha^0, y^0)}V_\alpha(y)\leq V_{\alpha^0}(y^0).
	\]
	Since the forms $\theta_1,\theta_2\ldots, \theta_r$ are positive, $\theta_{\alpha^0}$ is also positive. Thus $V_{\alpha^0}(y^0)=0$ and upper semicontinuity is trivial in this case. 
	
	Next, we treat the case when $|\alpha^0|<1$.
	In this case, there exists $\varepsilon>0$ such that  $\alpha^0+\varepsilon:=(\alpha^0_1+\varepsilon, \alpha^0_2+\varepsilon, \dots, \alpha^r+\varepsilon)$ lies in the interior of $\Box_L$. 
	Take a sufficiently small neighborhood $U_{\alpha_0, \varepsilon}$ in $\Box_L$ of $\alpha_0$ such that every point $\alpha$ in $U_{\alpha_0, \varepsilon}$ satisfies $\alpha\leq \alpha^0+\varepsilon$.
	It follows from Lemma \ref{lem:V_alpha_lim} $(i)$ that
	\[
	\limsup_{\Box_L\times Y\ni(\alpha, y)\to (\alpha^0, y^0)}\frac{V_\alpha(y)}{1-|\alpha|}=
	\limsup_{U_{\alpha_0, \varepsilon}\times Y\ni(\alpha, y)\to (\alpha^0, y^0)}\frac{V_\alpha(y)}{1-|\alpha|}\leq
	\limsup_{Y\ni y\to y^0}\frac{V_{\alpha^0+\varepsilon}(y)}{1-|\alpha^0+\varepsilon|}
	\]
	holds.
	By the upper semicontinuity of $\textstyle\frac{V_{\alpha^0+\varepsilon}}{1-|\alpha^0+\varepsilon|}$, we have
	\[
	\limsup_{\Box_L\times Y\ni(\alpha, y)\to (\alpha^0, y^0)}\frac{V_\alpha(y)}{1-|\alpha|}\leq
	\frac{V_{\alpha^0+\varepsilon}(y^0)}{1-|\alpha^0+\varepsilon|}.
	\]
	Letting $\epsilon \to 0$, we have 
	\[
	\limsup_{\Box_L\times Y\ni(\alpha, y)\to (\alpha^0, y^0)}\frac{V_\alpha(y)}{1-|\alpha|}\leq
	\frac{V_{\alpha^0}(y^0)}{1-|\alpha^0|},
	\]
	which shows the upper semicontinuity of the function $F(\alpha, y)/(1-|\alpha|)$ near $(\alpha^0, y^0)$. 
By multiplying by the continuous function $(\alpha, y)\mapsto 1-|\alpha|$, we have that $F$ itself is also upper semicontinuous. 
\end{proof}

\subsubsection{Integral formula}\label{subsection:integral}
In the following, we consider the local coordinates on $U_j^{(r+1)}$ defined as in \S 2.
Recall that we defined homogeneous fiber coordinates $[x_{j,1}:x_{j,2}: \cdots: x_{j,r+1}]$ on $U_j^{(r+1)}$. We define the fiber coordinate $z_1,\ldots, z_r$ on $U_j^{(r+1)}$ as $z_\lambda := {x_{j,\lambda}}/{x_{j,r+1}}$.
Then, we have $s_\lambda=z_\lambda$ ($1 \leq \lambda \leq r$) and $s_{r+1} =1$. We rewrite $\widehat{V}_B^\mathbb{Q}$ as follows:
\[\widehat{V}_B^\mathbb{Q} = \sup_{\substack{\ell_\lambda,  m \in \mathbb{Z} \\ \ell/m \in \Box_L}} \left[\frac{2\ell_1}{m}\log|z_1| + \cdots + \frac{2\ell_{r}}{m}\log|z_{r}| + (\varphi_{\ell/m})_B(y)\right],\]
where $\varphi_{\ell/m}$ is defined by $(\ell_1/m)\varphi_1 +(\ell_2/m)\varphi_2 + \cdots + (\ell_r/m)\varphi_{r} + (1-(\ell_1+\ell_2+\cdots + \ell_{r})/m)\varphi_{r+1}$ and $\phi_B$ denotes a local weight corresponding to a function $V_{\phi,B}$.
Here we regard $z_\lambda$ as a holomorphic function on $U_j^{(r+1)}$.

Define a volume form $dV$ on $X$ by setting
\[dV:= \frac{\pi^*\omega^n}{n!} \wedge \frac{e^{\varphi_1+\cdots+\varphi_{r+1}}}{(|z_1|^2e^{\varphi_1}+\cdots+|z_r|^2e^{\varphi_r}+e^{\varphi_{r+1}})^{r+1}} (\sqrt{-1})^{r}dz_1\wedge d\overline{z}_1 \wedge \ldots \wedge dz_{r} \wedge d\overline{z}_{r}\]
on $U_{j}^{(r+1)}$.
Here, $\omega$ is a fixed K\"{a}hler form on $Y$ as in \S \ref{subsection: outline}.
A simple calculation shows that this form extends to a smooth volume form on $X$.

In the proof of Theorem \ref{thm:P_bdl_main}, we need the following integral formula. Here we integrate a function $F=|z_1|^{2t_1}\cdots |z_{r}|^{2t_{r}} e^{-\varphi_{L}} = \displaystyle\frac{|z_1|^{2t_1}\cdots |z_{r}|^{2t_{r}}}{|z_1|^2e^{\varphi_1}+\cdots+|z_r|^2e^{\varphi_r}+e^{\varphi_{r+1}}}$ by using the volume form $dV$.

\begin{lemma}\label{lem:integration}
	\[\int_{(z_1,\ldots,z_{r}) \in \mathbb{C}^{r}} \frac{|z_1|^{2t_1}\cdots |z_{r}|^{2t_{r}} \cdot e^{\varphi_1+\cdots+\varphi_{r+1}}}{(|z_1|^2e^{\varphi_1}+\cdots+|z_r|^2e^{\varphi_r}+e^{\varphi_{r+1}})^{r+2}}(\sqrt{-1})^{r}dz_1\wedge d\overline{z}_1 \wedge \ldots \wedge dz_{r} \wedge d\overline{z}_{r} \]
	\[=(2\pi)^{r}\frac{\Gamma(1+t_1)\cdots\ \Gamma(1+t_{r})\Gamma(2-(t_1+\cdots+t_{r}))}{\Gamma(r+2)e^{\varphi_t}}.\]
\end{lemma}

\begin{proof}
	To make ideas clear, we first prove the case that $r=2$.
	In this case, the equation we want to prove is as follows:	
	\[\int_{(z_1,z_2) \in \mathbb{C}^2} \frac{|z_1|^{2t_1} |z_2|^{2t_2} \cdot e^{\phi_1+\phi_2+\phi_3}}{(|z_1|^2e^{\phi_1}+|z_2|^2e^{\phi_2}+e^{\phi_3})^{4}}(\sqrt{-1})^2dz_1\wedge d\overline{z}_1 \wedge dz_2 \wedge d\overline{z}_2 \]
	\[=(2\pi)^2\frac{\Gamma(1+t_1) \Gamma(1+t_2)\Gamma(2-(t_1+t_2))}{\Gamma(4)e^{\phi_t}},\]
	where $\phi_t:=t_1\phi_1 +t_2\phi_2 + (1-(t_1+t_2))\phi_3$.
	Write $z_1=a_1e^{i\theta_1}$ and $z_2=a_2e^{i\theta_2}$ and define $A_\lambda$ by $A_\lambda:=e^{\phi_\lambda}$ ($\lambda = 1,2,3$).
	Then we have 
	\begin{align*}
	&\int_{\mathbb{C}^{2}} \frac{|z_1|^{2t_1}|z_2|^{2t_2} \cdot A_1A_2A_3 }{(|z_1|^2 A_1 + |z_2|^2 A_2 +A_3)^{4}} (\sqrt{-1})^2dz_1\wedge d\overline{z}_1 \wedge dz_{2} \wedge d\overline{z}_{2}\\
	=&(2\pi)^{2}2^{2} \int_{\mathbb{R}^{2}_{>0}} \frac{a_1^{1+2t_1} a_{2}^{1+2t_{2}} A_1 A_2 A_3}{(a^2_1 A_1+ a^2_{2}A_{2}+A_3)^{4}}da_1da_{2}.
	\end{align*}
	Next, we use the following polar coordinates $(s, \theta)$:
	\begin{align*}
	s &:= \left(\frac{A_1}{A_3}a_1^2 + \frac{A_2}{A_3}a_2^2\right)^{1/2} \in \mathbb{R}_{\geq 0},\\
	\theta &:= {\rm Arctan} \left(\frac{\sqrt{A_2}a_2}{\sqrt{A_1}a_1}\right) \in [0, \pi/2].
	\end{align*}
	Note that these coordinates are written as $a_1 = \sqrt{A_3/A_1} \cdot s \cos \theta$ and $a_2 = \sqrt{A_3/A_2} \cdot s \sin \theta$.
	Then we have 
	\begin{align*}
	& \int_{\mathbb{R}^{2}_{>0}} \frac{a_1^{1+2t_1} a_{2}^{1+2t_{2}} A_1 A_2 A_3}{(a^2_1 A_1+ a^2_{2}A_{2}+A_3)^{4}}da_1da_{2}\\
	= &A_1^{-t_1}A_2^{-t_2}A_3^{t_1+t_2-1} \int_{s=0}^\infty \frac{s^{3+2t_1+2t_2}}{(s^2+1)^4}ds \int_{\theta=0}^{\pi/2} (\cos \theta)^{1+2t_1} (\sin \theta)^{1+2t_2} d\theta.
	\end{align*}
	We denote this value by $I$.
	We will calculate the integration in $s$ and $\theta$ respectively.
	
	First we consider the integration in $s$.
	To compute, we use the substitution $\sigma =s^2$. Then,
	\[ 
	\int_{s=0}^\infty \frac{s^{3+2t_1+2t_2}}{(s^2+1)^4}ds
	=\frac{1}{2}\int_{\sigma = 0}^{\infty} \frac{\sigma^{1+t_1+t_2}}{(\sigma+1)^4} d\sigma
	=\frac{1}{2}B(2+t_1+t_2,2-t_1-t_2).
	\]
	At the last equality, we use the formula \cite[5.12.3]{OLBC} for the beta function.
	
	Next, we consider the integration in $\theta$. By the formula \cite[5.12.2]{OLBC}, we have
	\[\int_{\theta=0}^{\pi/2} (\cos \theta)^{1+2t_1} (\sin \theta)^{1+2t_2} d\theta = \frac{1}{2}B(1+t_1,1+t_2).\]
	By \cite[5.12.1]{OLBC}, we have
	\[(2\pi)^{2}2^{2} I = (2\pi)^2A_1^{-t_1}A_2^{-t_2}A_3^{t_1+t_2-1} \frac{\Gamma(2-t_1-t_2)\Gamma(1+t_1)\Gamma(1+t_2)}{\Gamma(4)}.\]
	The proof in the case that $r=2$ is complete.
	
	Now we will prove the theorem in the general case.
	Since the proof is almost the same, we only explain the essential points.
	We use the coordinate change $z_\lambda=a_\lambda e^{i\theta_\lambda}$ and get the expression of $a_1,\ldots,a_r$. Then we use the $r$-dimensional polar coordinate
	\begin{align*}
	a_1 &= \sqrt{A_{r+1}/A_1} \cdot s \cos \theta_1\\
	a_2 &= \sqrt{A_{r+1}/A_2} \cdot s \sin \theta_1 \cos \theta_2\\
	a_3 &= \sqrt{A_{r+1}/A_3} \cdot s \sin \theta_1 \sin \theta_2 \cos \theta_3\\
	\vdots &\\
	a_{r-1} &= \sqrt{A_{r+1}/A_{r-1}} \cdot s \sin \theta_1 \sin \theta_2 \sin \theta_3 \cdots \sin \theta_{r-2} \cos \theta_{r-1}\\
	a_{r} &= \sqrt{A_{r+1}/A_{r}} \cdot s \sin \theta_1 \sin \theta_2 \sin \theta_3 \cdots \sin \theta_{r-2} \sin \theta_{r-1},\\
	\end{align*}
	where $A_\lambda = e^{\varphi_\lambda}$.
	The determinant of the Jacobian matrix is written as
	\[\sqrt{\frac{A_{r+1}}{A_1}\cdots \frac{A_{r+1}}{A_{r}}}s^{r-1}(\sin\theta_1)^{r-2}(\sin\theta_2)^{r-3} \cdots (\sin \theta_{r-2})^1.\]
	Finally we use formulae of the beta function to deduce the conclusion.	
\end{proof}

\subsubsection{Proof of Theorem \ref{thm:P_bdl_main} $(iv)$}
As we described in \S \ref{subsection: outline}, we shall prove
$V_{\theta_L,B} \leq \widehat{V}^{\mathbb{Q}}_B+C$.
Fix a global section $F \in H^0(X,L^m)$ with
$$\int_{X}|F|^2_{h^m_L}dV = 1.$$ 
We will decompose $F$ into orthogonal components using the following claim.
\begin{claim}
	The direct decomposition of $H^0(X,L^m)$ induced by the isomorphisms $H^0(X,L^m) = H^0(X,\mathcal{O}_{X}(m)) = H^0(Y, S^m (M_1 \otimes \cdots \otimes M_{r+1}) ) \cong \textstyle\bigoplus_{\ell_1+\cdots+\ell_{r+1} = m} H^0(Y,M_1^{\ell_1} \otimes \cdots \otimes M_{r+1}^{\ell_{r+1}} )$ is the orthogonal decomposition with respect to the $L^2$-norm defined by Hermitian metric $h^m_{L}$ of $L^{m}$ and the volume form $dV$. Here, $S^m E$ denotes the $m$-th symmetric tensor of $E$.
\end{claim}

\begin{proof}
	By the decomposition above, we have injective morphisms $H^0(Y,M_1^{\ell_1} \otimes \cdots \otimes M_{r+1}^{\ell_{r+1}} ) \to H^0(X,L^m)$ for every $(r+1)$-tuple of non-negative integers $\ell = (\ell_1,\ldots,\ell_{r+1})$ with $\ell_1+\cdots+\ell_{r+1} = m$. In the following, $M_1^{\ell_1} \otimes \cdots \otimes M_{r+1}^{\ell_{r+1}}$ will be denoted by $M^\ell$.
	This morphism maps $f_\ell \in H^0(Y,M^\ell)$ to $s_1^{\ell_1} \cdot s_2^{\ell_2} \cdots s_{r+1}^{\ell_{r+1}} \pi^* f_\ell$.
	
	We will prove that, for any $\ell=(\ell_1,\ldots,\ell_{r+1})$ and $\ell'=(\ell'_1,\ldots,\ell'_{r+1})$ with $\ell \neq \ell'$, two sections $\beta := s_1^{\ell_1} \cdot s_2^{\ell_2} \cdots s_{r+1}^{\ell_{r+1}} \pi^* f_\ell$ and $\beta':=s_1^{\ell'_1} \cdot s_2^{\ell'_2} \cdots s_{r+1}^{\ell'_{r+1}} \pi^* f_{\ell'}$ are orthogonal.
	We regard $\beta$ and $\beta'$ as holomorphic functions via the local trivialization.
	Then, by the equations $s_\lambda = z_\lambda$ ($\lambda = 1,2,\ldots,r$) and $s_{r+1} = 1$, it follows that
	\begin{align*}
	&\int_{X} \langle \beta, \beta' \rangle_{h^m_L} dV
	= \int_{X} (z_1^{\ell_1} \cdot z_2^{\ell_2} \cdots z_{r}^{\ell_{r}} \pi^* f_\ell) \cdot (\overline{z_1^{\ell'_1} \cdot z_2^{\ell'_2} \cdots z_{r}^{\ell'_{r}} \pi^* f_{\ell'}}) e^{-m \varphi_L} dV\\
	=& \int_{y \in Y} \left[  \int_{z \in \pi^{-1}(y)}  F(z)  (\sqrt{-1})^rdz_1\wedge d\overline{z}_1 \wedge \ldots \wedge dz_{r} \wedge d\overline{z}_{r} \right] \frac{\omega^n}{n!},
	\end{align*}
	where $$F(z) = \frac{(z_1^{\ell_1} \cdots z_{r}^{\ell_{r}} \pi^* f_\ell) \overline{(z_1^{\ell'_1} \cdots z_{r}^{\ell'_{r}} \pi^* f_{\ell'})}e^{-m \varphi_{L}}e^{\varphi_1+\cdots+\varphi_{r+1}}}{(|z_1|^2e^{\varphi_1}+\cdots+|z_r|^2e^{\varphi_r}+e^{\varphi_{r+1}})^{r+1}}.$$
	Write $z_\lambda=s_\lambda r^{i\theta_\lambda}$.
	Considering integration in $\theta_\lambda$'s, we have that it becomes 0 if $\ell \neq \ell'$.
	Therefore, two sections are orthogonal for different $\ell$ and $\ell'$.
\end{proof}

Let us decompose $F \in H^0(X,L^m)$ into the sum of the components $\beta_\ell = s_1^{\ell_1} \cdot s_2^{\ell_2} \cdots$ $s_{r+1}^{\ell_{r+1}}$ $\cdot \pi^* f_\ell$, where $f_\ell \in H^0(Y,M_1^{\ell_1} \otimes \cdots \otimes M_{r+1}^{\ell_{r+1}} )$, according to the orthogonal decomposition obtained in the claim.
By the orthogonality, we have
\[\int_X |\beta_\ell|^2_{h^m_{L}} dV \leq \int_X |F|^2_{h^m_{L}} dV (=1).\]
Next we estimate the norm of $f_\ell$ .
\begin{claim}
	There exist constants $C_1$ and $C_2$ independent of $m$ such that 
	\[\int_{y \in Y} |f_\ell(y)|^2 e^{-m(\varphi_{\ell/m})} \frac{\omega^n}{n!} \leq C_1 C_2^m \int_{X} |\beta_\ell|^2 e^{-m\varphi_{L}} dV,\]
	where $\varphi_{\ell/m}$ stands for $\textstyle\frac{\ell_1}{m}\varphi_1 +\textstyle\frac{\ell_2}{m}\varphi_2 + \cdots + \textstyle\frac{\ell_{r}}{m}\varphi_{r} + (1-\textstyle\frac{\ell_1+\cdots + \ell_{r}}{m})\varphi_{r+1}$ as in \S \ref{subsection:integral}.
\end{claim}

\begin{proof}
	Let
	$z^\ell:= z_1^{\ell_1} \cdot z_2^{\ell_2} \cdots z_{r}^{\ell_{r}} $. 
	Then we have $z^\ell \pi^* f_\ell = \beta_\ell$ under the trivialization.
	We estimate the right-hand side from below. We have that
	\begin{align*}
	\int_{X} |\beta_\ell|^2 e^{-m\varphi_{L}} dV = \int_{X} |z^\ell f_\ell(x)|^2 e^{-m\varphi_{L}} dV =\int_{y \in Y} |f_\ell(x)|^2 \left[ \int_{z \in \pi^{-1}(y)} |z^\ell|^{2} e^{-m\varphi_{L}}dP\right] \frac{\omega^n}{n!},
	\end{align*}
	where $dP$ is the measure on the fiber $\pi^{-1}(y)$ defined as  $$dP:=\displaystyle\frac{e^{\varphi_1+\cdots+\varphi_{r+1}}}{(|z_1|^2e^{\varphi_1}+\cdots+|z_r|^2e^{\varphi_r}+e^{\varphi_{r+1}})^{r+1}}(\sqrt{-1})^rdz_1\wedge d\overline{z}_1 \wedge \ldots \wedge dz_{r} \wedge d\overline{z}_{r}.$$
	By H\"{o}lder's inequality, we have that
	\[\int_{z \in\pi^{-1}(y)}|z^\ell|^{2/m} \cdot e^{-\varphi_{L}} dP \leq \left[\int_{\pi^{-1}(y)} |z^\ell|^{2}\cdot e^{-m \varphi_{L}} dP\right]^{\frac{1}{m}} \cdot \left[\int_{\pi^{-1}(y)} 1 \cdot dP\right]^{\frac{m-1}{m}}.\]
	A straightforward computation similar to the proof of Lemma \ref{lem:integration} shows that the value of the integral $\textstyle\int_{\pi^{-1}(y)} dP$ in the right-hand side is a constant independent of $y$, which we will denote by $I$. 
	By Lemma \ref{lem:integration}, the integral in the left-hand side is equal to
	\[(2\pi)^{r}\frac{\Gamma\left(1+\frac{\ell_1}{m}\right)\cdots\Gamma\left(1+\frac{\ell_{r}}{m}\right)\Gamma\left(2-\left(\frac{\ell_1+\cdots+\ell_{r}}{m}\right)\right)}{\Gamma(r+2)e^{\varphi_\ell}}.\]
	As $\Gamma(t)$ is bounded for $1 \leq t \leq 2$ from below, there exists a positive constant $C>0$ such that $\Gamma(t) \geq C$ for $1\leq t \leq 2$.
	Combining these estimates, we have that
	\[(2\pi)^{r}\frac{C^{r+1}e^{-\varphi_\ell}}{\Gamma(r+2)} \leq \left[\int_{z \in \pi^{-1}(y)} |z^\ell|^{2} e^{-m\varphi_{L}}dP\right]^{\frac{1}{m}} \cdot I^{\frac{m-1}{m}}.\]
	Thus we obtain
	\begin{align*}
	\int_{X} |\beta_\ell|^2 e^{-m\varphi_{L}} dV =\,& \int_{y \in Y} |f_\ell(y)|^2 \left[ \int_{z \in \pi^{-1}(y)} |z^\ell|^{2} e^{-m\varphi_{L}}dP \right] \frac{\omega^n}{n!}\\
	\geq\, & \int_{y \in Y} |f_\ell(y)|^2 I^{-(m-1)}\left[(2\pi)^{r}\frac{C^{r+1}e^{-\varphi_\ell}}{\Gamma(r+2)}\right]^m \frac{\omega^n}{n!}\\
	=\, & \frac{(2\pi)^{rm}C^{rm}I^{-(m-1)}}{\Gamma(r+2)^m}\int_{y \in Y} |f_\ell(y)|^2 e^{-m\varphi_\ell} \frac{\omega^n}{n!}.
	\end{align*}
\end{proof}

By the previous estimates and the definition of the Bergman-type metric, we have
\[\frac{1}{m}\log|f_\ell|^2 \leq (\varphi_\ell)_B + \frac{\log C_1 + m\log C_2}{m}.\]

To prove the desired inequality $V_{\theta_L,B} \leq \widehat{V}^\mathbb{Q}_B + C$, we will estimate the norm of the section $F = \textstyle\sum_{\ell} z^\ell \pi^*f_\ell$ from above. 
Assume $\ell=(\ell_1,\ldots,\ell_{r})$ satisfies $\ell/m \in  \Box_L$.
By the argument as in the proof of \cite[Proposition 2.5(2)]{K2}, we have
\[\frac{2}{m} \log |F| \leq \max_\ell \left[\frac{2\ell_1}{m}\log|z_1| + \cdots + \frac{2\ell_{r}}{m}\log|z_{r}| + (\phi_\ell/m)_B\right] + \frac{1}{m}\log C_1 + \log C_2.\]
Let $C:= \log C_1 + \log C_2$. Then the right-hand side is estimated by $\widehat{V}_B^\mathbb{Q} + C$ by definition.
Since $m$ is arbitrary, the supremum of the left-hand side over $m$ and $F$ is $V_{\theta_L,B}$, and the proof is complete.


\section{Proof of the main results}\label{section:prf}

In this section, we prove Theorem \ref{cor:ab} and Theorem \ref{thm:main}. 
Let $X, Y$, and $L$ be those as in Theorem \ref{thm:main} (in particular, we assume that $Y$ admits a holomorphic tubular neighborhood).
The idea of the proof is based on \cite{K2}: we first construct a new ``projective {space} bundle model'' $(\widetilde{X}, \widetilde{Y}, \widetilde{L})$ from $(X, Y, L)$, and construct a metric of $L$ with minimal singularities by using the metric on $\widetilde{L}$  as in \S \ref{section:P_bdl}. 
See also \S 3.2 for the relation between the models $(X, Y, L)$ and $(\widetilde{X}, \widetilde{Y}, \widetilde{L})$. 

\subsection{The projective {space} bundle model $(\widetilde{X}, \widetilde{Y}, \widetilde{L})$}

Let $X, Y$, and $L$ be as in Theorem \ref{thm:main}. 
Denote by $\widetilde{X}$ the total space of the projective {space} bundle $\mathbb{P}(\mathbb{I}_Y\oplus N_{Y/X}^*)$, by $\widetilde{Y}$ the subvariety $\mathbb{P}(\mathbb{I}_Y)$, and by $\widetilde{L}$ the line bundle $\mathcal{O}_{\mathbb{P}(\mathbb{I}_Y\oplus N_{Y/X}^*)}(1)\otimes \pi^*L|_Y$, where $\pi\colon \widetilde{X}\to Y$ is the natural projection. 
Note that $\widetilde{X}={\bf P}(\mathbb{I}_Y\oplus N_{Y/X})$ and $\widetilde{Y}={\bf P}(\mathbb{I}_Y)$. 
It is easily observed that $\widetilde{X}$ is a compactification of the normal bundle $N_{Y/X}$, and from this point of view, $\widetilde{Y}$ is regarded as a zero section of $N_{Y/X}$. 
Therefore, by the assumption on the existence of a holomorphic tubular neighborhood, 
there exists a neighborhood $V$ of $Y$ in $X$,$\widetilde{V}$ of $\widetilde{Y}$ in $\widetilde{X}$ and a biholomorphic map $i\colon \widetilde{V}\cong V$ with $i|_{\widetilde{Y}}=\pi|_{\widetilde{Y}}$. 

\begin{proposition}\label{prop:L_vs_Ltilde}
If one choose $V$ sufficiently small, one have that $i^*L\cong\widetilde{L}$.
\end{proposition}

The proof is based on \cite[\S 3]{K2}. 
First we prove the following proposition as a higher codimensional analogue of \cite[Proposition 3.1]{K2}: 

\begin{proposition}\label{rmk_metric}
By shrinking $V$ suitably, one have the following: \\
$(1)$ (a version of Rossi's theorem) The natural map $H^1(V, \mathcal{O}_{V})\to H^1(V, \mathcal{O}_{V}/I_{Y}^n)$ is injective for some $n\geq 1$, where $I_{Y}$ the defining ideal sheaf of $Y\subset V$. \\
$(2)$ If $H^1(Y, S^\ell N_{Y/X}^*)$ vanishes for every $\ell\geq 1$, then the groups ${\rm Pic}^0\,(V)$ and ${\rm Pic}^0\,(Y)$ are isomorphic.
\end{proposition}

\begin{proof}
See the proof of \cite[Proposition 3.1]{K2} (We intrinsically use Rossi's theorem \cite[Theorem 3]{R}.
Here we remark that, by Lemma \ref{lem:V_psdconv_with_max_cpt_Y} below, we may assume that $V$ is a strongly pseudoconvex domain which has $Y$ as a maximal compact set). 
\end{proof}

\begin{lemma}\label{lem:V_psdconv_with_max_cpt_Y}
There exists a strongly pseudoconvex holomorphic tubular neighborhood $V$ of $Y$ which has $Y$ as a maximal compact analytic set. 
\end{lemma}

\begin{proof}
As $Y$ admits a holomorphic tubular neighborhood, it is sufficient to show the lemma by assuming $X=\widetilde{X}$. Take a $C^\infty$ Hermitian metric $h_{N_\lambda^{-1}}$ on $N_\lambda^{-1}$ with positive curvature. 
Taking a local coordinate $y$ of $Y$ and pulling it back by $\pi$, we regard $(z, y)=(z_1, z_2, \dots, z_r, y)$ as a local coordinates system of $X$, where $z_\lambda$ is a fiber coordinate of $N_\lambda$. 
By considering the sublevel set of the function $\Phi\colon N_{Y/X}\to \mathbb{R}$ defined by 
\[
\Phi(z, y):=\sum_{\lambda=1}^r|z_\lambda|^2 e^{\varphi_\lambda(y)}, 
\]
where $\varphi_\lambda$ is the local weight of $h_{N_\lambda^{-1}}$, the lemma follows.
\end{proof}

\begin{proof}[Proof of Proposition \ref{prop:L_vs_Ltilde}]
First note that
\[
K_Y^{-1}\otimes S^\ell N_{Y/X}^*=\bigoplus_{\alpha\in\mathbb{Z}_{\geq 0}^r, |\alpha|=\ell}K_Y^{-1}\otimes \left(\textstyle\bigotimes_{\lambda=1}^r N_\lambda^{-\alpha_\lambda}\right)
\]
holds. 
As $K_Y^{-1}\otimes N_\lambda^{-1}$ and $N_\mu$ are positive for each $\lambda, \mu=1, 2, \dots, r$, 
it follows from Kodaira's vanishing theorem that $H^1(Y, S^\ell N_{Y/X}^*)$ vanishes for each $\ell\geq 1$.
Thus, by Proposition \ref{rmk_metric}, 
it is sufficient to show that $\widetilde{L}|_{\widetilde{Y}}\cong\pi|_{\widetilde{Y}}^*L$ holds. 
The line bundle $\mathcal{O}_{\mathbb{P}(\mathbb{I}_Y\oplus N_{Y/X}^*)}(1)$ corresponds to the divisor $\mathbb{P}(N_{Y/X}^*)\subset\mathbb{P}(\mathbb{I}_Y\oplus N_{Y/X}^*)$, which does not intersect $\widetilde{Y}$. Therefore we have $\widetilde{L}|_{\widetilde{Y}}=\mathcal{O}_{\mathbb{P}(\mathbb{I}_Y\oplus N_{Y/X}^*)}(1)|_{\widetilde{Y}}\otimes \pi|_{\widetilde{Y}}^*L=\pi|_{\widetilde{Y}}^*L$. 
\end{proof}


\subsection{Minimal singular metrics on $L$ and $\widetilde{L}$}

Let $X, Y$, and $L$ be those in Theorem \ref{thm:main}, and let $\widetilde{X}, \widetilde{Y}, \widetilde{L}, V$, and $\widetilde{V}$ be those in the previous subsection. 
Here we prove the following: 

\begin{proposition}\label{prop:min_sings_comp}
Let $h$ be a metric of $L$ with minimal singularities and $\widetilde{h}$ be a metric of $\widetilde{L}$ with minimal singularities. Then $h|_V\sim_{\rm sing}(i^{-1})^*\widetilde{h}|_{\widetilde{V}}$ holds at each point in $V$. 
\end{proposition}

The proof of Proposition \ref{prop:min_sings_comp} is based on the ``maximum construction technique" which is also used in the proof of \cite[Theorem 1.2]{K2}. 
Fix a $C^\infty$ metric $h_\infty$ of $L$ and denote by $\theta$ the curvature tensor $\Theta_{h_\infty}$. 
Set $\varphi_V:=\sup\{\varphi\in PSH(V, \theta|_V)\mid \varphi\leq 0\ {\rm on}\ V\}$ and $\varphi_X:=\sup\{\varphi\in PSH(X, \theta)\mid \varphi\leq 0\ {\rm on}\ X\}$. 
We first show the following: 

\begin{lemma}\label{lem:X_min_vs_V_min}
It holds that $\varphi_V\sim_{\rm sing}\varphi_X$. 
In particular, the restriction of a metric of $L$ with minimal singularities to $V$ has singularities equivalent to the metric $h_{\infty}|_V\cdot e^{-\varphi_V}$. 
\end{lemma}

\begin{proof}
As the inequality $\varphi_X|_V\leq \varphi_V$ is easily obtained, all we have to do is to show the existence of a constant $C$ with $\varphi_V\leq \varphi_X|_V+C$. 
As $L$ is big and ${\bf B}(L)=Y$, 
we can take an integer $m\geq 1$ and sections $f_1, f_2, \dots, f_\ell\in H^0(X, L^m)$ such that the common zero of these sections is $Y$ \cite[2.1.21]{Laz}. 
Denote by $h_a$ the Bergman type metric on $L$ constructed from $f_1, f_2, \dots, f_\ell$ (see Example \ref{ex:bergman_type_sing_metric} ). 
Define the function $\varphi_a$ by $h_a=h_\infty\cdot e^{-\varphi_a}$. 
We may assume that $\varphi_a\leq 0$ holds on $V$. 
Fix a relatively compact open neighborhood $V_0\Subset V$ of $Y$ and set 
$C_1:=-\textstyle\inf_{V\setminus V_0}\varphi_a$. 
Consider a $\theta$-plurisubharmonic function $\widehat{\varphi}:=\max\{\varphi_V-C_1, \varphi_a\}$. 
As $\varphi_V-C_1\leq -C_1\leq \varphi_a$ holds on $V\setminus V_0$, we have $\widehat{\varphi}=\varphi_a$ on each point in $V\setminus V_0$. 
Thus we can extend $\widehat{\varphi}$ to whole $X$ by defining $\widehat{\varphi}(x):=\varphi_a(x)$ for each $x\in X\setminus V_0$. 
It is clear from the construction that $\widehat{\varphi}\in PSH(X, \theta)$. 
Set $C_2:=\textstyle\max_X \widehat{\varphi}$. 
Then, as $\widehat{\varphi}-C_2\leq 0$, we obtain that $\widehat{\varphi}-C_2\leq \varphi_X$. Therefore it holds that 
$\varphi_V-C_1-C_2\leq \widehat{\varphi}-C_2\leq \varphi_X$. 
\end{proof}

\begin{proof}[Proof of Proposition \ref{prop:min_sings_comp}]
Take a $C^\infty$ Hermitian metric $h_\infty$ on $L$ and $\widetilde{h}_\infty$ on $\widetilde{L}$
with $h_\infty|_V= i^*\widetilde{h}_\infty$ (here we used Proposition \ref{prop:L_vs_Ltilde}). 
By Lemma \ref{lem:X_min_vs_V_min}, it follows that $\varphi_V\sim_{\rm sing}\varphi_X$, where $\varphi_V$ and $\varphi_X$ are as above. 
Set $\varphi_{\widetilde{V}}:=\sup\{\varphi\in PSH(\widetilde{V}, \Theta_{\widetilde{h}_\infty}|_{\widetilde{V}})\mid \varphi\leq 0\ {\rm on}\ \widetilde{V}\}$ and $\varphi_{\widetilde{X}}:=\sup\{\varphi\in PSH(\widetilde{X}, \Theta_{\widetilde{h}_\infty})\mid \varphi\leq 0\ {\rm on}\ \widetilde{X}\}$. 
By the arguments in \S 3.2, we can apply Lemma \ref{lem:X_min_vs_V_min} also to the projective {space} bundle model $(\widetilde{X}, \widetilde{L})$ to obtain that $\varphi_{\widetilde{V}}\sim_{\rm sing}\varphi_{\widetilde{X}}$. 
As $h_\infty|_V= i^*\widetilde{h}_\infty$, we have that $\varphi_V=i^*\varphi_{\widetilde{V}}$. 
Therefore it follows that $\varphi_X\sim_{\rm sing}i^*\varphi_{\widetilde{X}}$. 
\end{proof}


\subsection{Proof of Theorem \ref{thm:main}}\label{subsec:mainthmproof}
Theorem \ref{thm:main} follows from 
Theorem \ref{thm:P_bdl_main} and Proposition \ref{prop:min_sings_comp}. 
\qed



\section{A sufficient condition for the existence of a holomorphic tubular neighborhood and Proof of Theorem \ref{cor:ab}}\label{section:tubular}

Let $X$ be a complex manifold and let $Y\subset X$ be a compact complex submanifold of codimension $r$. In this section, we investigate when $Y$ admits a holomorphic tubular neighborhood $V$ in $X$. 
In particular, we here study a higher codimensional analogue of Grauert's theorem (\cite{G}, the case of $r=1$, see also Theorem \ref{thm:Grauert} below). 

\subsection{A higher codimensional analogue of Grauert's theorem}

In this subsection, we show the following: 

\begin{proposition}\label{prop:grauert_higher_codim}
Assume that $N_{Y/X}$ admits a direct decomposition $N_{Y/X}=N_1\oplus N_2\oplus \cdots\oplus N_r$ into $r$ negative line bundles. 
Assume also that $H^1(E, \mathcal{O}_{{\bf P}(N_{Y/X})}(\nu))=0$ and $H^1(E, T_E\otimes \mathcal{O}_{{\bf P}(N_{Y/X})}(\nu))=0$ hold for every $\nu\geq 1$, where $E$ is the total space of the projective {space} bundle ${\bf P}(N_{Y/X})$. 
Then $Y$ admits a holomorphic tubular neighborhood. 
\end{proposition}

Note that Proposition \ref{prop:grauert_higher_codim} is the Grauert's theorem when $r=1$. 
We will consider a blow-up $p\colon W\to X$ of $X$ along $Y$ and apply the Grauert's theorem to $E\subset W$ to show this proposition, where we are regarding $E$ as the exceptional divisor (see \cite[Proposition 12.4]{D}). 
For this purpose, we first show the following: 

\begin{lemma}\label{lem:hol_tub_nbhd_or_vs_bup}
Assume that $E$ admits a holomorphic tubular neighborhood in $W$. 
Then $Y$ also admits a holomorphic tubular neighborhood in $X$. 
\end{lemma}

\begin{proof}
Denote by $\widetilde{Y}$ the zero section of $\pi\colon N_{Y/X}\to X$. 
Take a neighborhood $\widetilde{V}$ of $\widetilde{Y}$. 
We denote by $\widetilde{W}$ the blow-up $\widetilde{p}\colon \widetilde{W}\to \widetilde{V}$ of $\widetilde{V}$ along $\widetilde{Y}$ and by $\widetilde{E}\subset \widetilde{W}$ the exceptional set. 
By the assumption (and by shrinking $X$ if necessary), we may assume that there exists a biholomorphic map $F\colon \widetilde{W}\to W$ with $F|_{\widetilde{E}}=q|_{\widetilde{E}}$, where $q\colon\mathcal{O}_{{\bf P}(N_{Y/X})}(-1)\to E (={\bf P}(N_{Y/X}))$ is the natural projection (here we are regarding $\widetilde{W}$ as a neighborhood of the zero section $E$ of the line bundle $\mathcal{O}_{{\bf P}(N_{Y/X})}(-1)$, see also \cite[Proposition 12.4]{D}). 

First, let us construct a holomorphic map $g\colon V\to Y$ with $g|_Y={\rm id}_Y$ such that the diagram 
\[\xymatrix{
\widetilde{W} \ar[r]^F \ar[d]^{\rm } \ar@{}[dr]|\circlearrowleft &W \ar[r]^p \ar[d]^{} \ar@{}[dr]|\circlearrowleft & V \ar[d]^{g} \\
\widetilde{E} \ar[r]^{q|_{\widetilde{E}}} &E \ar[r]^{\rm } & Y & \\
}\]
is commutative. 
It is clear that such function $g$ is uniquely determined in the set-theoretic sense. 
It also follows from a standard argument that this $g$ is a continuous map. 
It is also clear that $g|_{V\setminus Y}$ is biholomorphic and that $g|_Y={\rm id}_Y$. 
Thus the existence of the holomorohic map $g$ follows from Riemann's extension theorem (see Lemma \ref{lem:hol_ext_conti} below). 

Next we show that the biholomorphic map $f\colon \widetilde{V}\setminus \widetilde{Y}\cong V\setminus Y$ induced by 
$F|_{\widetilde{W}\setminus \widetilde{E}}$ extends to the biholomorphism $\widehat{f}\colon \widetilde{V}\cong V$ with $\widehat{f}|_{\widetilde{Y}}=\pi|_{\widetilde{Y}}$. 
Note that, by construction, the fibration structures $\pi|_{\widetilde{V}}\colon\widetilde{V}\to\widetilde{Y}$ and $g\colon V\to Y$ are preserved by $f$. 
Therefore, by a simple topological argument, it follows that there uniquely exists a continuous map $\widehat{f}\colon \widetilde{V}\to V$ with $\widehat{f}|_{\widetilde{V}\setminus\widetilde{Y}}=f$, and that this $\widehat{f}$ satisfies $\widehat{f}|_{\widetilde{Y}}=\pi|_{\widetilde{Y}}$. 
The regularity of $\widehat{f}$ and $\widehat{f}^{-1}$ is shown again by Lemma \ref{lem:hol_ext_conti} below. 
\end{proof}

\begin{lemma}\label{lem:hol_ext_conti}
Let $M$ and $N$ be complex manifolds, let $Z\subset M$ be a submanifold with codimension greater than or equal to $1$, and let $h\colon M\to N$ be a continuous map. 
Assume that $h|_{M\setminus Z}$ is holomorphic. 
Then $h$ is holomorphic on $M$. 
\end{lemma}

\begin{proof}
Take a point $z\in Z$ and an open ball $U'\subset N$ with coordinate system $\eta=(\eta_1, \eta_2, \dots, \eta_n)$ around $h(z)$ ($n:={\rm dim}\,N$). 
We may assume that $U'=\{|\eta|<\varepsilon\}$ for some $\varepsilon>0$. 
Take a sufficiently small neighborhood $U\subset M$ of $z$ so that $U\subset h^{-1}(U')$ (and thus $h(U)\subset U'$). 
In what follows, we show the lemma by replacing $M$ with $U$, $N$ with $U'$, and $h$ with $h|_{U}$ (in particular, we are regarding $N$ as an open ball of $\mathbb{C}^n$). 
It is sufficient to show each function $h_\lambda$ is holomorphic on $z$, where $h=(h_1, h_2, \cdots, h_n)$ is the decomposition by the coordinate system $\eta=(\eta_1, \eta_2, \dots, \eta_n)$. 
As $h_\lambda$ is continuous (and thus it is locally bounded), we may assume that the $L^2$-norm of $h_\lambda|_{U\setminus Z}$ is bounded by shrinking $U$ if necessary. 
Therefore it follows from Riemann's extension theorem that we can extend $h_\lambda|_{U\setminus Z}$ to a holomorphic function $\widehat{h}_\lambda\colon U\to \mathbb{C}$. 
Since $U\setminus Z\subset U$ is a dense subset, we conclude that $h_\lambda=\widehat{h}_\lambda$, which proves the lemma. 
\end{proof}

By Lemma \ref{lem:hol_tub_nbhd_or_vs_bup}, all we have to do is to investigate when $E$ admits a holomorphic tubular neighborhood in $W$. 
We apply the following Grauert's theorem to this problem: 

\begin{theorem}[{\cite[Sats 7, p. 363]{G}, see also \cite[Theorem 4.4]{CM}}]\label{thm:Grauert}
Let $M$ be a complex manifold and let $Z\subset M$ be a strongly exceptional subvariety of pure codimension $1$. 
Assume that $H^1(Z, N_{Z/M}^{-\nu})=0$ and $H^1(Z, T_Z\otimes N_{Z/M}^{-\nu})=0$ hold for every $\nu\geq 1$. 
Then $Z$ has a holomorphic tubular neighborhood in $M$. 
\end{theorem}

\begin{proof}[Proof of Proposition \ref{prop:grauert_higher_codim}]
By the assumption, Lemma \ref{lem:hol_tub_nbhd_or_vs_bup} and Theorem \ref{thm:Grauert}, all we have to do is to show that $E\subset W$ is an exceptional subset (in the sense of Grauert). 
By \cite[Theorem 4.9, 6.12]{Lau}, \cite[Satz 8, p. 353]{G}, and \cite[Lemma 11]{HR} (see also \cite[Theorem 3.6]{CM}), it is sufficient to see the following two conditions: 
$(i)$ $N_{E/W}$ is negative, and 
$(ii)$ $\mathcal{O}_W/I_E^2\cong \mathcal{O}_{\widetilde{W}}/I_{\widetilde{E}^2}$, where $I_E\subset \mathcal{O}_W$ and $I_{\widetilde{E}}\subset \mathcal{O}_{\widetilde{W}}$ are the defining ideal sheaves of $E$ and $\widetilde{E}$, respectively. 
$(i)$ follows from $N_{E/W}^{-1}=\mathcal{O}_{{\bf P}(N_{Y/X})}(1)=\mathcal{O}_{\mathbb{P}(N_1^{-1}\oplus N_2^{-1}\oplus\cdots\oplus N_r^{-1})}(1)$ and the assumption that $N_\lambda$ is negative. 
$(ii)$ follows from the condition $H^1(E, T_E\otimes \mathcal{O}_{{\bf P}(N_{Y/X})}(1))=0$ (see \cite[Proposition 1.10, 1.11]{CM}). 
\end{proof}


\subsection{A sufficient condition for the existence of a holomorphic tubular neighborhood}

In this subsection, we show the following lemma as an application of Proposition \ref{prop:grauert_higher_codim}: 

\begin{lemma}\label{lem:a_suff_cond_for_hol_tub_nbhd}
Let $X$ be a complex manifold and let $Y$ be a compact complex submanifold. 
Assume that $N_{Y/X}$ admits a direct decomposition $N_{Y/X}=N_1\oplus N_2\oplus \cdots\oplus N_r$ into $r$ negative line bundles. 
Assume also the following three conditions:\\
$(i)$ $N_\lambda\cong N_\mu$ for each $\lambda$ and $\mu$, \\
$(ii)$ $N_\lambda^{-1}\otimes K_Y^{-1}\otimes T_Y$ is Nakano positive, and \\
$(iii)$ $N_\lambda^{-1}\otimes K_Y^{-1}$ is ample for each $\lambda$. \\
Then $Y$ admits a holomorphic tubular neighborhood in $X$. 
\end{lemma}

Note that, when $T_Y$ is holomorphically trivial, conditions $(ii)$ and $(iii)$ are automatically satisfied. 

\begin{proof}
By Proposition \ref{prop:grauert_higher_codim}, it is sufficient to show $H^1(E, \mathcal{O}_{{\bf P}(N_{Y/X})}(\nu))=0$ and $H^1(E, T_E\otimes \mathcal{O}_{{\bf P}(N_{Y/X})}(\nu))=0$ for every $\nu\geq 1$, where $E:={\bf P}(N_{Y/X})$. 
Note that it follows from condition $(i)$ that $E\cong Y\times {\bf P}^r$. 
By the relative Euler sequence 
\[
0\to \mathbb{I}_{E}\to p|_E^*N_{Y/X}\otimes\mathcal{O}_{{\bf P}(N_{Y/X})}(1)\to T_{E/Y}\to 0, 
\]
it turns out that it is sufficient to show the following four vanishing assertions: \linebreak
$H^1(E, \mathcal{O}_{{\bf P}(N_{Y/X})}(\nu))=0$, 
$H^2(E, \mathcal{O}_{{\bf P}(N_{Y/X})}(\nu))=0$, 
$H^1(E, \mathcal{O}_{{\bf P}(N_{Y/X})}(\nu+1)\otimes p|_E^*N_\lambda)=0$, and 
$H^1(E, \mathcal{O}_{{\bf P}(N_{Y/X})}(\nu)\otimes p|_E^*T_Y)=0$ for each $\nu\geq 1$. 
By Nakano's vanishing theorem, the problem is reduced to show Nakano positivity for the following three vector bundles: 
$K_E^{-1}\otimes\mathcal{O}_{{\bf P}(N_{Y/X})}(1)$, 
$K_E^{-1}\otimes\mathcal{O}_{{\bf P}(N_{Y/X})}(2)\otimes p|_E^*N_\lambda$, and 
$K_E^{-1}\otimes\mathcal{O}_{{\bf P}(N_{Y/X})}(1)\otimes p|_E^*T_Y$. 
As 
\[
K_E^{-1}\cong p|_E^*(N^r\otimes K_Y^{-1})\otimes \mathcal{O}_{{\bf P}(N_{Y/X})}(r)
\]
holds ($N:=N_1\cong N_2\cong\dots\cong N_r$), these three bundles are written as
$\mathcal{O}_{{\bf P}(N_{Y/X})}(r+1)\otimes p|_E^*(N^r\otimes K_Y^{-1})$, 
$\mathcal{O}_{{\bf P}(N_{Y/X})}(r+2)\otimes p|_E^*(N^{r+1}\otimes K_Y^{-1})$, and 
$\mathcal{O}_{{\bf P}(N_{Y/X})}(r+1)\otimes p|_E^*(T_Y\otimes N^r\otimes K_Y^{-1})$, respectively.
In what follows, we show the Nakano positivity for these three bundles. 

First let us note that $\mathcal{O}_{{\bf P}(N_{Y/X})}(1)\otimes p|_E^*N={\rm Pr}_1^*\mathcal{O}_{{\bf P}^r}(1)$, where ${\rm Pr}_1$ is the first projection $E\cong {\bf P}^r\times Y\to {\bf P}^r$. 
Let us denote by $h$ the metric on this line bundle which is the pull-back of the Fubini-Study metric by ${\rm Pr}_1$. 
By tensoring $h$ and a metric on $N^{-1}\otimes K_Y^{-1}$ with positive curvature, we have the positivity for the bundles $\mathcal{O}_{{\bf P}(N_{Y/X})}(r+1)\otimes p|_E^*(N^r\otimes K_Y^{-1})={\rm Pr}_1^*\mathcal{O}_{{\bf P}^r}(r+1)\otimes p|_E^*(N^{-1}\otimes K_Y^{-1})$ and $\mathcal{O}_{{\bf P}(N_{Y/X})}(r+2)\otimes p|_E^*(N^{r+1}\otimes K_Y^{-1})={\rm Pr}_1^*\mathcal{O}_{{\bf P}^r}(r+2)\otimes p|_E^*(N^{-1}\otimes K_Y^{-1})$. 

Finally we show the Nakano positivity for $F:=\mathcal{O}_{{\bf P}(N_{Y/X})}(r+1)\otimes p|_E^*(T_Y\otimes N^r\otimes K_Y^{-1})=\mathcal{O}_{{\bf P}(\mathbb{I}_Y)}(r+1)\otimes p|_E^*(T_Y\otimes N^{-1}\otimes K_Y^{-1})$. 
Take a metric $h'$ on $T_Y\otimes N^{-1}\otimes K_Y^{-1}$ with Nakano positive curvature.  
Then $h^{r+1}\otimes p|_E^*h'$ is a metric on $F$, with curvature
$(r+1)\Theta_h\otimes {\rm Id}_F+p|_E^*\Theta_{h'}$, which is easily seen to be Nakano positive. 
\end{proof}

\subsection{Proof of Theorem \ref{cor:ab}}

By Lemma \ref{lem:a_suff_cond_for_hol_tub_nbhd}, $Y$ admits a holomorphic tubular neighborhood in $X$. 

By Theorem \ref{thm:main}, there exists a metric $h_{{\rm min}, L}$ with minimal singularities whose local weight $\varphi_{{\rm min}, L}$ is written in the form 
\[
\varphi_{{\rm min}, L}(z, y)=\log \max_{\alpha \in \Box_L} |z^\alpha|^2 e^{(\varphi_\alpha)_e (y)}+ O(1). 
\]
By choosing metrics $e^{\varphi_\alpha (y)}$ as in \cite[\S 2.2]{K1}, 
we obtain that $(\varphi_\alpha)_e (y)=\varphi_\alpha (y)$ holds and $\varphi_\alpha(y)$ depends continuously on $(y, \alpha)$, which proves the assertion. 
\qed


\section{Examples}\label{section:eg}

\subsection{Nakayama's example}

Nakayama's example $(X, L, Y)$ is the example which admits no Zariski decomposition even after modifications \cite[IV, \S 2.6]{N} (see also \cite[\S1]{K1}). 
In this example, the manifold $X$ is a total space of the projective {space} bundle $\pi\colon X:=\mathbb{P}(M_1\oplus M_2\oplus M_3)\to Y$ over an abelian surface $Y$, where $M_1$ and $M_2$ are ample line bundles on $Y$ and $M_3$ is a line bundle on $Y$. 
The line bundle $L$ is the inverse of the tautological line bundle: i.e. $L:=\mathcal{O}_{\mathbb{P}(M_1\oplus M_2\oplus M_3)}(1)$. 
Then the stable base locus of $L$ is the subset $\mathbb{P}(M_3)\subset X$, which we are regarding as $Y$ here. 
As it is clearly observed, this example $(X, L, Y)$ is a special case of those in \S 3. Therefore we can apply Theorem \ref{thm:P_bdl_main} to this example. 
In particular, by using the metrics as in the proof of Theorem \ref{cor:ab}, we can reprove the main result in \cite{K1}. 

\subsection{Zariski's example and its higher (co-)dimensional analogues and proof of Theorem \ref{cor:higher_Zariski}}\label{section:zariski}

In \cite[\S 4.2]{K2}, the second-named author applied its main result (=Theorem \ref{cor:ab}, \ref{thm:main} of this paper for $r=1$) to Zariski's and Mumford's example $(X, L, Y)$, in which $L$ is nef and big however not semi-ample, and showed the semi-positivity of $L$ (i.e. the existence of a $C^\infty$ Hermitian metric on $L$ with semi-positive curvature). 
Here we construct an example which can be regarded as a higher-codimensional analogue of Zariski's example and apply Theorem \ref{cor:ab} to it. In what follows, we only consider the case of $r=2$ for simplicity. 

Take two general quadric surfaces $Q_1$ and $Q_2$ in $\mathbb{P}^3$. 
Then we may assume that the intersection $C:=Q_1\cap Q_2$ is a smooth elliptic curve and $Q_1$ and $Q_2$ intersects transversally along $C$. 
Fix $N$ points $p_1, p_2, \dots, p_N$ in $C$. 
Denote by $\pi\colon X:={\rm Bl}_{\{p_1, p_2, \dots, p_N\}}\mathbb{P}^3\to \mathbb{P}^3$ the blow-up of 
$\mathbb{P}^3$ at these $N$ points, 
by $Y$ the strict transform of $C$, 
by $D_1$ and $D_2$ the the strict transform of $Q_1$ and $Q_2$, respectively, 
by $E_\lambda$ the exceptional divisor $\pi^{-1}(p_\lambda)$ for each $\lambda$, 
by $E$ the divisor $\textstyle\sum_{\lambda=1}^NE_\lambda$, 
and by $H$ the pull-back $\pi^*\mathcal{O}_{\mathbb{P}^3}(1)$. 
Note that $D_\lambda\in |2H-E|$. 

Let us consider the line bundle $L:=\mathcal{O}_X(H+D_1)=\mathcal{O}_X(3H-E)$ on $X$. 
As $H$ is big and $D_1$ is effective, $L$ is also big. 
It is also observed that ${\rm Bs}\,|L|\subset Y$ holds, since $H$ is base point free and ${\rm Bs}\,|H|\subset Y$ by construction. 
As a simple computation shows that the intersection number $(L. Y)$ is equal to $12-N$, 
we conclude that $L$ is nef if and only if $12\geq N$. 

First let us consider the case of $N=12$. 
In this case, we may assume that $L|_Y$ is a general (and thus non-torsion) element of ${\rm Pic}^0\,(Y)$ by choosing $p_1, p_2, \dots, p_{12}$ generically. 
Then it is easily observed that ${\bf B}(L)=Y$ holds, and therefore that $L$ is not semi-ample, hoverer $L$ is nef and big. 
In this sense, we can regard this example $(X, Y, L)$ as an analogue of Zariski's example with $r=2$. 
As $D_1$ and $D_2$ intersects transversally along $Y$, we obtain the decomposition 
\[
N_{Y/X}=N_{D_1/X}|_Y\oplus N_{D_2/X}|_Y=\mathcal{O}_X(D_1)|_{D_1}|_Y\oplus \mathcal{O}_X(D_2)|_{D_2}|_Y
=\mathcal{O}_X(D_1)|_Y\oplus \mathcal{O}_X(D_2)|_Y
\]
By denoting $N_\lambda:=\mathcal{O}_X(D_\lambda)|_Y$ for each $\lambda=1, 2$, it holds that $N_1\cong N_2$, 
$(D_\lambda. Y)=2(H. Y)-(E. Y)=8-12<0$, and ${\rm deg}_YL|_Y\otimes N_\lambda^{-1}=0-(8-12)>0$. 
Therefore we can apply Theorem \ref{cor:ab}, Corollary \ref{cor:restr_min}, and also Lemma \ref{lem:a_suff_cond_for_hol_tub_nbhd} to our $(X, L, Y)$. 
By Corollary \ref{cor:restr_min}, we have that $h_{{\rm min}, L}|_Y$ is bounded, where $h_{{\rm min}, L}$ is a metric of $L$ minimal singularities (here we use that fact that $L|_Y$ admits a $C^\infty$ Hermitian metric with zero curvature, since $L|_Y$ is a flat line bundle). 
Therefore we can conclude that $h_{{\rm min}, L}$ is bounded. 
Note that we can moreover show that the semi-positivity of $L$ (i.e. that we can choose $h_{{\rm min}, L}$ as a $C^\infty$ Hermitian metric) by applying Lemma \ref{lem:a_suff_cond_for_hol_tub_nbhd} and use the ``regularized maximum construction'' technique as in \cite[Corollary 3.4]{K2}. 

Next let us consider the case of $N>12$. 
In this case, $L$ is not nef . 
By the argument as above, we also have ${\bf B}(L)=Y$, ${\rm deg}\,N_\lambda=8-N<0$, and ${\rm deg}\,(L|_Y\otimes N_\lambda^{-1})=(12-N)-(8-N)=4>0$. Thus we can apply Theorem \ref{cor:ab} to $(X, Y, L)$ also in this case. 
As the computation shows that
\begin{eqnarray}
\Box_L
&=&\left\{\alpha=(\alpha_1, \alpha_2)\in \mathbb{R}^2_{\geq 0} \left| \frac{N-12}{N-8}\leq|\alpha| \leq 1\right.\right\}, \nonumber
\end{eqnarray}
it follows from Theorem \ref{cor:ab} that the local weight function $\varphi_{{\rm min}, L}$ of a metric $h_{{\rm min}, L}$ with minimal singularities can be written as 
\[
\varphi_{{\rm min}, L}(z, y)=\log \max_{\alpha \in \Box_L}\prod_{\lambda=1}^r |z_\lambda|^{2\alpha_\lambda} +O(1)
=\frac{N-12}{N-8}\cdot \log (|z_1|^2+|z_2|^2) +O(1)
\]
on a neighborhood of any point of $Y$, where $y$ is a coordinate of $Y$ and $z=(z_1, z_2, \dots, z_r)$ is a system of local defining functions of $Y$. 
In particular, in this case, $\varphi_{{\rm min}, L}$ has analytic singularities along $Y$. 

Note that similar example can be constructed in general dimension by considering some points blow-up of a del Pezzo manifold of degree $1$ (see \cite{F} for example. For the choice of the counterpart of the divisors $D_\nu$'s above, see \cite[\S 6.3]{K3}). 

\subsection{An example in {\cite{BEGZ}}}\label{section:BEGZ}

The above two examples satisfies Condition \ref{cond:main} $(ii)$ and the condition that $Y$ admits a holomorphic tubular neighborhood.
On the other hand, the example $(X, Y, L)$ in \cite[Example 5.4]{BEGZ} does not satisfy these conditions. 
In \cite{BEGZ}'s example, a metric of $L$ with minimal singularities is unbounded and actually has singularities along $Y$ (i.e.\ local weight function equals to $-\infty$ on $Y$), however the Lelong number of the local weight is $0$ for every point in $X$ (see also \cite[Example 4.2]{K2}). 
In particular, the conclusion of Theorem \ref{thm:main} does not hold in this example. 


\ \\
\vskip3mm
{\bf Acknowledgment. }
The authors would like to thank Dr.\ Tatsuya Miura for his helpful suggestions on the change of variables in the proof of Lemma \ref{lem:integration}.
The first author is supported by Program for Leading Graduate Schools, MEXT, Japan.
He is also supported by the Grant-in-Aid for Scientific Research (KAKENHI No.15J08115).
The second author is supported by the Grant-in-Aid for Scientific Research (KAKENHI No.28-4196).



\end{document}